\documentclass{article}
\usepackage{amsmath, amsthm, amssymb}
\usepackage[dvips]{graphicx}
\usepackage{enumerate}
\def\authorsPS{}
\newcommand\authPS[8]{\ifnum\authPSc<1\def\authPSc{1}\else, \fi{\large\sffamily #1 #2}%
\edef\authorsPS{\authorsPS \par\vskip 1mm\noindent #1 #2:\hskip 5mm #3, #4, #5, #6, #7, #8}}

\newtheorem{theorem}{Theorem}[section]
\newtheorem{lemma}[theorem]{Lemma}

\newtheorem{proposition}[theorem]{Proposition}
\theoremstyle{definition}
\newtheorem{definition}[theorem]{Definition}
\theoremstyle{remark}
\newtheorem{remark}[theorem]{Remark}

\def\authPSc{0}
\newcommand{\beq}[1]{%\marginpar{\footnotesize\sf #1}%
\begin{equation}\label{#1}}
\newcommand{\eeq}{\end{equation}}
\newcommand{\req}[1]{{\rm(\ref{#1})}}
\newcommand{\Ten}{\stackrel{\sim}{\otimes}}
\newcommand{\hten}{{\mathfrak H}}
\newcommand{\law}{\stackrel{\cal L}{\longrightarrow}}
\newcommand{\Nt}{\lfloor nt \rfloor}
\newcommand{\Btilde}{\widetilde{B}_{\frac jn}}

\begin{document}

% !!!!!!!!!!!!!!!!!!!!!!!!!!!!!!!!!!!!!!!!!!!!!!!!!!!!!!!!!!!!!!!!!!!!!!!
% !!!                   START HERE                                   !!!
% !!!!!!!!!!!!!!!!!!!!!!!!!!!!!!!!!!!!!!!!!!!!!!!!!!!!!!!!!!!!!!!!!!!!!!!

%% !!  FULL TITLE OF PAPER
% example :

\title{On Simpson's rule and fractional Brownian motion with H = 1/10}

\author{Daniel Harnett and David Nualart\thanks{
D. Nualart is supported by the NSF grant DMS1208625. \newline
  }   \\
   {\small \em Univ. of Wisconsin - Stevens Point and University of Kansas} }
\date{}
\maketitle

\begin{abstract}
We consider stochastic integration with respect to fractional Brownian motion (fBm) with $H < 1/2$.  The integral is constructed as the limit, where it exists, of a sequence of Riemann sums.  A theorem by Gradinaru, Nourdin, Russo \& Vallois (2005) holds that a sequence of Simpson's rule Riemann sums converges in probability for a sufficiently smooth integrand $f$ and when the stochastic process is fBm with $H > 1/10$.  For the case $H = 1/10$, we prove that the sequence of sums converges in distribution.  Consequently, we have an It\^o-like formula for the resulting stochastic integral.  The convergence in distribution follows from a Malliavin calculus theorem that first appeared in Nourdin and Nualart (2010).
\end{abstract}

\section{Introduction}

Let $B = \{ B^H_t, t\ge 0\}$ be a fractional Brownian motion (fBm), that is, $B$ is a centered Gaussian process with covariance given by
\beq{cov_fbm}{\mathbb E}\left[ B_s B_t\right]:=R(s,t) =\frac 12\left( s^{2H} + t^{2H} - |t-s|^{2H}\right),\eeq
for $s,t \ge 0$, where $H\in(0,1)$ is the Hurst parameter.  For a smooth function $f:{\mathbb R}\to {\mathbb R}$, we take the `Simpson's rule' Riemann sum with uniform partition,
\[ S^S_n(t) := \sum_{j=0}^{\Nt -1} \frac 16\left(f'(B_{\frac jn}) + 4 f'\left( (B_{\frac jn}+B_{\frac{j+1}{n}})/2\right)+f'(B_{\frac{j+1}{n}})\right)\left( B_{\frac{j+1}{n}} - B_{\frac jn}\right).\]
It can be shown (see \cite{GNRV}, or Section 3.1) that this sequence of sums converges in probability when $B$ is fBm with $H > 1/10$, but in general it does not converge in probability when $H \le 1/10$.  In this paper, we consider the particular case of $H = 1/10$, and show that $S^S_n(t)$ does converge weakly to a random variable.  More precisely, Theorem 3.3 shows that, conditioned on the path $ \{B_s, s\le t\}$, \beq{mainSSnt}  S^S_n(t) \law  f(B_t) - f(0)  + \frac{\beta}{2880}\int_0^t f^{(5)}(B_s) dW_s,\eeq
where $W_t$ is a standard Brownian motion, independent of $B$, and $\beta$ is a constant defined in Theorem 3.3.  This result allows us to write the change-of-variable formula
\beq{cov_110} f(B_t) \stackrel{\cal L}{=} f(0) + \int_0^t f'(B_s)d^SB_s - \frac{\beta}{2880}\int_0^t f^{(5)}(B_s) dW_s,\eeq
where the differential $d^SB_s$ denotes the limit of the Simpson's rule sum.

Conditional convergence in distribution follows from a central limit theorem given in Section 2 (Theorem 2.3).  This is a new version of a theorem that first appeared in Nourdin and Nualart (2010) \cite{NoNu}.  This theorem uses Malliavin calculus, and applies to a random vector with components in the form of Malliavin divergence integrals.  After proving Theorem 2.3, the main task in proving \req{cov_110} is to verify the conditions of Theorem 2.3, which are relatively long and technical.

\subsection{Background.}
Assuming a uniform partition, the classical Stratonovich stochastic integral is defined as 
\beq{Stratonov} \int_0^t f'(B_s)d^\circ B_s = \lim_{n\to\infty} S^T_n(t) := \lim_{n\to\infty} \sum_{j=0}^{\Nt -1} \frac 12\left( f'(B_{\frac jn})+f'(B_{\frac{j+1}{n}})\right)\left( B_{\frac{j+1}{n}} - B_{\frac jn}\right),\eeq
provided that limit exists.  It has been shown that this limit exists in probability when $B$ is a fBm with $H > 1/6$ , but does not, in general converge in probability for $H \le 1/6$ (see \cite{ChNu, GNRV, NoRevSwan}, also Section 3.1).  Subsequently, it was proved in \cite{NoRevSwan} that for $H=1/6$, \req{Stratonov} does converge in law to a random variable that includes a Wiener-It\^o integral, that is, as $n\to\infty$
\[S_n^T(t) \law f(B_t) - f(0) +\gamma \int_0^t f^{(3)} (B_s)dW_s,\]
where $\gamma$ is a known constant and $W$ is a standard Brownian motion, independent of $B$.  Hence, there is the change-of-variable formula
\beq{cov_16} f(B_t) \stackrel{\cal L}{=} f(0) + \int_0^t f'(B_s)d^\circ B_s -\gamma\int_0^t f^{(3)}(B_s)~dW_s.\eeq

The reader will recognize that \req{Stratonov} is the Riemann sum corresponding to the `Trapezoidal rule' of basic calculus.  It is certainly possible to generalize to other types of Riemann sums.  The `Midpoint' sum,
\[ \sum_{j=1}^{\left\lfloor \frac{nt}{2}\right\rfloor} f'(B_{\frac{2j-1}{n}})\left( B_{\frac{2j}{n}} - B_{\frac{2j-2}{n}}\right),\]
can be shown to converge in probability for fBm with $H > 1/4$ (see \cite{Swanson}).  The end point case $H=1/4$ was considered in papers by Burdzy and Swanson \cite{BuSw}, and Nourdin and R\'eveillac \cite{NoRev}.  These papers proved the change-of-variable formula
\beq{cov_14} f(B_t) \stackrel{\cal L}{=} f(0) +\int_0^t f'(B_s) d^\star B_s + \theta \int_0^t f''(B_s) dW_s,\eeq
where $\theta$ is a constant, $W$ is a scaled Brownian motion, independent of $B$, and the notation $d^\star B_s$ denotes the integral arising from the midpoint sum.

\subsection{Extensions.}
Following the results \req{cov_16} and \req{cov_14}, the present authors also wrote papers on the cases $H = 1/4$ and $H=1/6$ \cite{HaNu_14, HaNu_16}.  These papers contained alternate proofs of \req{cov_14} and \req{cov_16}, using Malliavin calculus and a version of Theorem 2.3.  An interesting difference in the present paper, is that the sum $S^S_n(t)$ converges conditionally to a random variable that is actually the sum of two, independent Gaussian random variables.  In the cases considered in \cite{HaNu_14, HaNu_16}, there was only a single random term.  In those prior papers, we also showed that the results could be extended to other Gaussian processes sufficiently similar to fBm, for example, bifractional Brownian motion with $HK=1/6$ in the case of \req{cov_16}.  It was also shown that the Midpoint and Trapezoidal Riemann sums converge as functions in the Skorohod space ${\mathbf D}[0,\infty)$, by proving that the sums converge in the sense of finite-dimensional distributions.  We expect that similar extensions could be applied to the present Theorem 3.3, but we have not pursued this in the present paper.

We also expect that the techniques of this paper could be applied to the `Milne's rule' sum for the case $H =1/14$, see Proposition 3.1.

\medskip
The organization of this paper is as follows:  in Section 2, we give a brief description of the Malliavin calculus definitions and identities that will be used.  We also discuss properties of fBm, and prove the central limit theorem which will be applied for the main result.  In Section 3, after a brief introduction we state and prove the main result, which is Theorem 3.3.  Finally, Section 4 contains proofs of three of the longer lemmas from Section 3. 

\section{Notation and Theory}
Let $f:{\mathbb R} \to {\mathbb R}$ be a function and $N$ be a Gaussian random variable with mean zero and variance $\sigma^2$.  We say that $f$ satisfies {\em moderate growth conditions} if there exist constants $A, B$, and $\alpha < 2$ such that
$|f(x)| \le Ae^{B|x|^\alpha}$.  Note that this implies ${\mathbb E}\left[ |f(N)|^p\right] < \infty$ for all $p \ge 1$. We use the symbol ${\mathbf 1}_{[0,t]}$ to denote the indicator function for a real interval $[0,t]$.  The symbol $C$ denotes a generic positive constant, which may vary from line to line. In general, the value of $C$ will depend on and the growth conditions of a test function $f$ and the properties of a stochastic process $B$.

\subsection{Elements of Malliavin Calculus.}
Following is a brief description of some identities that will be used in the paper.  The reader may refer to  \cite{Nualart} for detailed coverage of this topic.  
Let $Z = \{ Z(h), h\in\cal{H}\}$ be an {\em isonormal Gaussian process} on a probability space $( \Omega, {\cal F}, P )$, and indexed by a real separable Hilbert space $\cal{H}$.  That is, $Z$ is a family of Gaussian random variables such that ${\mathbb E}[ Z(h)] =0$ and ${\mathbb E}\left[Z(h)Z(g)\right] = \left< h,g\right>_{\cal{H}}$ for all $h,g \in\cal{H}$.  We will assume that $\cal F$ is the $\sigma-$algebra generated by $Z$.
  
For integers $q \ge 1$, let ${\cal H}^{\otimes q}$ denote the $q^{th}$ tensor product of ${\cal H}$, and ${\cal H}^{\odot q}$ denote the subspace of symmetric elements of ${\cal H}^{\otimes q}$.  We will also use the notation $\bigotimes_{i=1}^r h_i$ to denote an arbitrary tensor product, with the convention that $\bigotimes_{i=1}^0$ is the empty set.  

Let $\{ e_n, n\ge 1\}$ be a complete orthormal system in ${\cal H}$.  For functions $f, g \in {\cal H}^{\odot q}$ and $p\in\{0, \dots, q\}$, we define the $p^{th}$-order contraction of $f$ and $g$ as that element of ${\cal H}^{\otimes 2(q-p)}$ given by
\beq{contract} f\otimes_p g = \sum_{i_1, \dots , i_p=1}^\infty \left< f, e_{i_1} \otimes \cdots\otimes e_{i_p}\right>_{{\cal H}^{\otimes p}} \otimes \left< g, e_{i_1} \otimes \cdots\otimes e_{i_p}\right>_{{\cal H}^{\otimes p}}\eeq
where $f\otimes_0 g = f\otimes g$ and $f\otimes_q g = \left< f,g\right>_{{\cal H}^{\otimes q}}$.  While $f, g$ are symmetric, the contraction $f \otimes_q g$ may not be.  We denote its symmetrization by $f \widetilde \otimes_q g$.

Let ${\cal H}_q$ be the $q^{th}$ Wiener chaos of $Z$, that is, the closed linear subspace of $L^2(\Omega)$ generated by the random variables $\{ H_q(Z(h)), h \in {\cal H}, \|h \|_{\cal H} = 1 \}$, where $H_q(x)$ is the $q^{th}$ Hermite polynomial, defined as
$$H_q (x) = {(-1)^q}e^{\frac{x^2}{2}}\frac{d^q}{dx^q}e^{-\frac{x^2}{2}},$$
and we follow the convention of Hermite polynomials with unity as a leading coefficient.  For $q \ge 1$, it is known that the map 
\beq{Hmap} I_q(h^{\otimes q}) = H_q(Z(h))\eeq
provides a linear isometry between ${\cal H}^{\odot q}$ (equipped with the modified norm $\sqrt{q!}\| \cdot\|_{{\cal H}^{\otimes q}}$) and ${\cal H}_q$, where $I_q(\cdot)$ is the generalized Wiener-It\^o multiple stochastic integral.  By convention, ${\cal H}_0 = \mathbb{R}$ and $I_0(x) = x$.  It follows from \req{Hmap} and the properties of the Hermite polynomials that for $f\in{\cal H}^{\odot p},$ $g \in {\cal H}^{\odot q}$ we have
\beq{I_q_norm} {\mathbb E}\left[ I_p(f) I_q(g)\right] = \begin{cases}p!\left< f,g\right>_{{\cal H}^{\otimes p}}&\text{ if }\; p=q\\0&\text{otherwise}\end{cases}.\eeq

Let $\cal S$ be the set of all smooth and cylindrical random variables of the form $F = g(Z(\phi_1), \dots, Z(\phi_n))$, where $n \ge 1$; $g: {\mathbb R}^n \to {\mathbb R}$ is an infinitely differentiable function with compact support, and $\phi_i \in {\cal H}$. The Malliavin derivative of $F$ with respect to $Z$ is the element of $L^2(\Omega; {\cal H})$ defined as
$$DF = \sum_{i=1}^n \frac{\partial g}{\partial x_i}(Z(\phi_1), \dots, Z(\phi_n)) \phi_i.$$
By iteration, for any integer $q >1$ we can define the $q^{th}$ derivative $D^qF$, which is an element of $L^2(\Omega; {\cal H}^{\odot q})$.

We let ${\mathbb D}^{q,2}$ denote the closure of $\cal S$ with respect to the norm $\| \cdot \|_{{\mathbb D}^{q,2}}$ defined as 
$$\| F \|_{{\mathbb D}^{q,2}}^2 = {\mathbb E}\left[ F^2\right] + \sum_{i=1}^q {\mathbb E}\left[ \| D^iF \|_{{\cal H}^{\otimes i}}^2 \right].$$
More generally, for any Hilbert space $V$, let ${\mathbb D}^{k,p}(V)$ denote the corresponding Sobolev space of $V-$valued random variables.

We denote by $\delta$ the Skorohod integral, which is defined as the adjoint of the operator $D$.  A random element $u \in L^2(\Omega; {\cal H})$ belongs to the domain of $\delta$, Dom $\delta$, if and only if,
$$\left| {\mathbb E}\left[ \left< DF, u\right>_{\cal H}\right] \right| \le c_u \|F\|_{L^2(\Omega)}$$
for any $F \in {\mathbb D}^{1,2}$, where $c_u$ is a constant which depends only on $u$.  If $u \in $ Dom $\delta$, then the random variable $\delta (u) \in L^2(\Omega)$ is defined for all $F \in {\mathbb D}^{1,2}$ by the duality relationship,
\[{\mathbb E}\left[ F\delta(u) \right] = {\mathbb E}\left[ \left< DF, u \right>_{\cal H} \right].\]
This is sometimes called the Malliavin integration by parts formula.  We iteratively define the multiple Skorohod integral for $q \ge 1$ as $\delta (\delta^{q-1}(u))$, with $\delta^0(u) = u$.  For this definition we have,
\beq{duality}{\mathbb E}\left[ F\delta^q(u) \right] = {\mathbb E}\left[ \left< D^qF, u \right>_{{\cal H}^{\otimes q}} \right],\eeq
where $u \in$ Dom $\delta^q$ and $F \in {\mathbb D}^{q,2}$.  The adjoint operator $\delta^q$ is an integral in the sense that for a (non-random) $h \in {\cal H}^{\odot q}$, we have $\delta^q(h) = I_q(h)$. 

The following results will be used extensively in this paper.  The reader may refer to \cite{NoNu} and \cite{Nualart} for proofs and details.
\begin{lemma}  Let $q \ge 1$ be an integer, and $r, j, k >0$ be integers.

\begin{enumerate}[(a)]
\item Assume $F \in {\mathbb D}^{q,2}$, $u$ is a symmetric element of {\em Dom}  $\delta^q$, and $\left< D^rF, \delta^j (u) \right>_{{\cal H}^{\otimes r}} \in L^2(\Omega; {\cal H}^{\otimes q-r-j})$ for all $0 \le r+j \le q.$  Then $\left< D^rF, u \right>_{{\cal H}^{\otimes r}} \in \,${\em Dom }$\delta^r$ and
$$ F \delta^q(u) = \sum_{r=0}^q \binom{q}{r} \delta^{q-r}\left(\left< D^rF, u \right>_{{\cal H}^{\otimes r}}\right).$$

\item  Suppose that $u$ is a symmetric element of ${\mathbb D}^{j+k,2}({\cal H}^{\otimes j})$.  Then we have,
$$D^k \delta^j (u) = \sum_{i=0}^{j \wedge k} i!\binom{k}{i} \binom{j}{i} \delta^{j-i}\left(D^{k-i}u\right).$$

\item  {\em Meyer Inequality:} Let $p > 1$ and integers $k \ge q \ge 1$.  Then for any $u \in {\mathbb D}^{k,p}({\cal H}^{\otimes q})$,
$$ \left\| \delta^q (u)\right\|_{{\mathbb D}^{k-q,p}} \le c_{k,p} \left\| u \right\|_{{\mathbb D}^{k,p}}({\cal H}^{\otimes q}),$$
where $c_{k,p}$ is a constant.  %In particular,
%$$\left\| \delta^q (u) \right\|^2_{L^2(\Omega)} = {\mathbb E}\left[ \delta^q(u)^2\right] \le
%c_{q,2} \sum_{z=0}^q  {\mathbb E}\left[ \left\| D^{z}u\right\|_{{\cal H}^{\otimes z}}^2\right].$$ 
\item Let $u\in {\cal H}^{\odot p}$ and $v \in {\cal H}^{\odot q}$.  Then
\[ \delta^p(u) \delta^q(v) = \sum_{z=0}^{p\wedge q} z! \binom pz \binom qz \delta^{p+q-2z}(u \otimes_z v),\]
where $\otimes_z$ is the contraction operator defined in \req{contract}.
\end{enumerate}\end{lemma}

\subsection{A convergence theorem.}
\begin{definition} %Stable Convergence
Assume $F_n$ is a sequence of $d-$dimensional random variables defined on a probability space $(\Omega, {\cal F}, P)$, and $F$ is a $d-$dimensional random variable defined on $(\Omega, {\cal G}, P)$, where ${\cal F} \subset {\cal G}$.  We say that $F_n$ {\em converges stably} to $F$ as $n \to \infty$, if, for any continuous and bounded function $f:\,{\mathbb R}^d \to {\mathbb R}$ and ${\mathbb R}$-valued, ${\cal F}-$measurable random variable $M$, we have
$$\lim_{n\to \infty} {\mathbb E} \left(f(F_n) M\right) = {\mathbb E}\left(f(F) M\right).$$
\end{definition}

%%Main Convergence Theorem
The first version of the following central limit theorem appeared in \cite{NoNu}.  In \cite{HaNu_14}, we extended this to a multi-dimensional version, where the sequence was a vector of $d$ components all in the same Wiener chaos.  For our present paper, we need a slight modification.  In this version, we lay out conditions for stable convergence of a sequence of vectors, where the vector components are not necessarily in the same Wiener chaos.
\begin{theorem} 
Let $d \ge 1$ be an integer, and $q_1, \dots, q_d$ be positive integers with $q^* = \max\{ q_1, \dots, q_d\}$.  Suppose that $F_n$ is a sequence of random variables in $\mathbb{R}^d$ of the form $F_n = \left( \delta^{q_1} (u_n^1), \dots, \delta^{q_d}(u_n^d) \right)$, where each $u_n^i$ is a $\mathbb{R}-$valued symmetric function in $\mathbb{D}^{2q^*, 2q_i}({\cal H}^{\otimes {q_i}})$.  Suppose that the sequence $F_n$ is bounded in $L^1(\Omega)$ and that:
\begin{enumerate}[(a)]
\item  $\left< u_n^j, \bigotimes_{\ell =1}^m (D^{a_\ell}F_n^{j_\ell}) \otimes h \right>_{{\cal H}^{\otimes q}}$ converges to zero in $L^1(\Omega)$ for all  integers $1 \le j, j_\ell \le d$, all integers $1 \le a_1, \dots, a_m, r \le q_j-1$  such that $a_1 + \cdots + a_m + r = q_j$; and all $h \in {\cal H}^{\otimes r}$.
\item  For each $1 \le i, j \le d$, $\left< u_n^i, D^{q_i}F_n^i\right>_{{\cal H}^{\otimes {q_i}}}$ converges in $L^1(\Omega)$ to a nonnegative random variable $s_i^2$,  and for $i \neq j$, $\left< u_n^i, D^{q_i}F_n^j\right>_{{\cal H}^{\otimes {q_i}}}$ converges to zero in $L^1(\Omega)$.
\end{enumerate}
Then $F_n$ converges stably to a random vector in ${\mathbb R}^d$,  whose components each have independent Gaussian law ${\cal N} (0, s_i^2)$ given $Z$. \end{theorem}

\medskip 
\begin{proof}
This proof mostly follows that given in \cite{HaNu_14}, except in that case there was only a single value of $q$.  We use the conditional characteristic function.  Given any $h_1, \dots h_m \in {\cal H}$, we want to show that the sequence 
$$\xi_n = \left(F_n^1, \dots, F_n^d, Z(h_1), \dots, Z(h_m) \right)$$
converges in distribution to a vector $\left(F_\infty^1, \dots F_\infty^d,  Z(h_1), \dots, Z(h_m) \right)$, where, for any vector $\lambda \in {\mathbb R}^d$, $F_\infty$ satisfies
\beq{Ch_gen}{\mathbb E}\left(e^{i\lambda^T F_\infty} | Z(h_1), \dots, Z(h_m)\right) = \exp\left(-\frac{1}{2}\lambda^T S \lambda \right),\eeq
where $S$ is the diagonal $d \times d$ matrix with entries $s_i^2$.

\medskip  Since $F_n$ is bounded in $L^1(\Omega)$, the sequence $\xi_n$ is tight in the sense that for any $\varepsilon > 0$, there is a $K>0$ such that $P\left( F_n\in [-K,K]^d\right) > 1-\varepsilon$, which follows from Chebyshev inequality.  
Dropping to a subsequence if necessary, we may assume that $\xi_n$ converges in distribution to a limit $\left(F_\infty^1, \dots F_\infty^d,  Z(h_1), \dots Z(h_m) \right)$.  Let $Y := g\left(Z(h_1), \dots, Z(h_m)\right)$, where $g \in {\cal C}^\infty_b ({\mathbb R}^m)$, and consider $\phi_n(\lambda) = \phi(\lambda, \xi_n) := {\mathbb E}\left(e^{i \lambda^T F_n} Y\right)$ for $\lambda \in {\mathbb R}^d$.  The convergence in law of $\xi_n$ implies that for each $1\le j\le d$:
\beq{Partial_g1} \lim_{n\to\infty} \frac{\partial \phi_n}{\partial \lambda_j} = \lim_{n\to\infty} i {\mathbb E}\left(F_n^j e^{i \lambda^T F_n} Y\right) =  i {\mathbb E}\left(F_\infty^j e^{i \lambda^T F_\infty} Y\right),\eeq
where convergence in distribution follows from a truncation argument applied to $F_n^j$.

On the other hand, using the duality property of the Skorohod integral and the Malliavin derivative:
$$ \frac{\partial \phi_n}{\partial\lambda_j} = i {\mathbb E}\left(\delta^{q_j}(u_n^j) e^{i\lambda^T F_n} Y \right) = i{\mathbb E}\left(\left< u_n^j,D^{q_j}\left(e^{i\lambda^T F_n}Y\right)\right>_{\hten^{\otimes {q_j}}}\right)$$

$$ = i \sum_{a=0}^{q_j} \binom{q_j}{a} {\mathbb E}\left(\left<u_n^j, D^a\left(e^{i\lambda^T F_n}\right) \Ten D^{q_j-a}Y\right>_{\hten^{\otimes q_j}}\right) $$

\beq{Expand_2g} = i \left\{ {\mathbb E} \left< u_n^j, Y D^{q_j} e^{i\lambda^T F_n} \right>_{{\cal H}^{\otimes q_j}} + \sum_{a=0}^{q_j-1} \binom{q_j}{a} {\mathbb E} \left<u_n^j, D^a e^{i \lambda^T F_n} \Ten D^{q_j-a} Y\right>_{{\cal H}^{\otimes q_j}} \right\} \eeq

By condition (a), we have that $\left< u_n^j, D^a e^{i\lambda^T F_n} \Ten D^{q_j-a}Y\right>_{{\cal H}^{\otimes q_j}}$ converges to zero in $L^1(\Omega)$ when $a < q_j$, so the sum term vanishes as $n \to \infty$, and this leaves
$$\lim_{n \to \infty} i {\mathbb E}\left< u_n^j , YD^qe^{i \lambda^T F_n} \right>_{{\cal H}^{\otimes q_j}} \; = \lim_{n\to \infty} i \sum_{k=1}^d {\mathbb E} \left( i\lambda_k e^{i \lambda^T F_n} \left<u_n^j, YD^{q_j} F_n^k \right>_{{\cal H}^{\otimes q_j}} \right)$$
$$ = - {\mathbb E}\left( \lambda_j e^{i \lambda^T F_\infty} s_j^2 Y \right)$$
because the lower-order derivatives in $D^{q_j} e^{i\lambda^T F_n}$ also vanish by condition (a), and cross terms ($j \neq k$) terms vanish by condition (b).  Combining this with \req{Partial_g1}, we obtain:
$$i {\mathbb E}\left(F_{\infty}^j e^{i \lambda \cdot F_\infty} Y \right) = -\lambda_j {\mathbb E} \left(e^{i \lambda \cdot F_\infty} s_{j}^2Y\right).$$

This leads to the PDE system:
$$\frac{\partial}{\partial \lambda_j} {\mathbb E} \left(e^{i \lambda^T F_\infty} | Z(h_1), \dots, Z(h_m)\right) = -  \lambda_j s_{j}^2 {\mathbb E} \left(e^{i \lambda^T F_\infty} | Z(h_1), \dots, Z(h_m)\right)$$
which has unique solution \req{Ch_gen}. \end{proof}

\medskip
\begin{remark}  It suffices to impose condition (a) for $h \in {\cal S}_0$, where ${\cal S}_0$ is a total subset of ${\cal H}^{\otimes r}$.\end{remark}

\begin{remark}
Suppose $F_n$ is the vector sequence $(F_n, G_n)$, where $F_n = \delta^p(u_n)$ and $G_n = \delta^q (v_n)$.  Then to satisfy Theorem 2.3, $F_n$ and $G_n$ must be bounded in $L^1(\Omega)$, and the following terms must tend to zero in $L^1(\Omega)$:
\begin{enumerate}
\item $\left< u_n, h\right>_{{\cal H}^{\otimes p}}$ and $\left< v_n, g\right>_{{\cal H}^{\otimes q}}$, for arbitrary $h \in{\cal H}^{\otimes p}$ and $g\in {\cal H}^{\otimes q}$, respectively.
\item $\left< u_n, \bigotimes_{i=1}^s D^{a_i} F_n \bigotimes_{i=s+1}^r D^{a_i}G_n \otimes h\right>_{{\cal H}^{\otimes p}}$, where $0 \le a_i < p$, $a_1 + \cdots + a_r < p$, and $h \in {\cal H}^{\otimes p-(a_1+\cdots +a_r)}$; and $\left< u_n, \bigotimes_{i=1}^s D^{a_i} F_n \bigotimes_{i=s+1}^r D^{a_i}G_n \right>_{{\cal H}^{\otimes p}}$, where $0 \le a_i < p$ and $a_1 + \cdots + a_r =p$.
\item $\left< v_n, \bigotimes_{i=1}^s D^{a_i} F_n \bigotimes_{i=s+1}^r D^{a_i}G_n \otimes h\right>_{{\cal H}^{\otimes q}}$, where $0 \le a_i < q$, $a_1 + \cdots + a_r < q$, and $h \in {\cal H}^{\otimes q-(a_1+\cdots +a_r)}$; and $\left< v_n, \bigotimes_{i=1}^s D^{a_i} F_n \bigotimes_{i=s+1}^r D^{a_i}G_n \right>_{{\cal H}^{\otimes q}}$, where $0 \le a_i < q$ and $a_1 + \cdots + a_r =q$.
\item $\left< u_n, D^p G_n\right>_{{\cal H}^{\otimes p}}$ and $\left< v_n, D^q F_n\right>_{{\cal H}^{\otimes q}}$.
\end{enumerate}
Then for condition (b), the following two terms must converge in $L^1(\Omega)$ to nonnegative random variables:  
 $\left< u_n, D^p F_n\right>_{{\cal H}^{\otimes p}}$ and $\left< v_n, D^q G_n\right>_{{\cal H}^{\otimes q}}$.
\end{remark}

\subsection{Fractional Brownian motion.}%%%%%%%%%%%%%%%%%%%%%%%%%
For some $T>0$, let $B = \{B_t^H, 0\le t \le T\}$ be a fractional Brownian motion with Hurst parameter $H$. That is, $B$ is a centered Gaussian process with covariance $R(s,t)$ given in \req{cov_fbm}.  Let ${\cal E}$ denote the set of ${\mathbb R}$-valued step functions on $[0,T]$.  We then let $\hten$ be the Hilbert space defined as the closure of $\cal E$ with respect to the inner product
\[ \left< {\mathbf 1}_{[0,s]}, {\mathbf 1}_{[0,t]}\right>_\hten = R(s,t).\]
The mapping ${\mathbf 1}_{[0,t]} \mapsto B_t$ can be extended to a linear isometry between $\hten$ and the Gaussian space spanned by $B$.  In this way, $\{ B(h), h\in\hten\}$ is an isonormal Gaussian process as in Section 2.1.

For an integer $n \ge 2$, we consider a uniform partition of $[0, \infty)$ given by $\{ j/n, j \ge 1\}$. Define the following notation:

\begin{itemize}
%\item $\varepsilon_t = {\mathbf 1}_{[0,t]}$
\item $\Delta B_{\frac jn} = B_{\frac{j+1}{n}} - B_{\frac jn}$, and $\widetilde B_{\frac jn} = \frac 12\left(B_{\frac jn} + B_{\frac{j+1}{n}}\right)$
\item $\partial_{\frac jn} = {\mathbf 1}_{[\frac jn, \frac{j+1}{n}]}$, $\varepsilon_t = {\mathbf 1}_{[0,t]}$, and $\widetilde \varepsilon_{\frac jn} = \frac 12 \left({\mathbf 1}_{[0,\frac{j}{n}]}+{\mathbf 1}_{[0,\frac{j+1}{n}]}\right) = \varepsilon_{\frac jn} + \frac 12 \partial_{\frac jn}$.
\end{itemize}
Assume $H < 1/2$.  The following fBm properties follow from \req{cov_fbm}.
\begin{enumerate}[(B.1)]
\item ${\mathbb E}\left[ \Delta B_{\frac jn}^2\right] =  \left< \partial_{\frac jn}, \partial_{\frac jn}\right>_\hten = n^{-2H}$.
\item ${\mathbb E}\left[ \Delta B_{\frac jn}\Delta B_{\frac{j+1}{n}}\right]= \left< \partial_{\frac jn}, \partial_{\frac{j+1}{n}}\right>_\hten = (2^{2H}-2)/2n^{2H}$.
\item If $|k-j| \ge 2$, $\left|{\mathbb E}\left[ \Delta B_{\frac jn}\Delta B_{\frac{k}{n}}\right]\right| = \left|\left< \partial_{\frac jn}, \partial_{\frac kn}\right>_\hten\right| \le Cn^{-2H} |j-k|^{2H-2}$, where the constant $C$ does not depend on $j$.
\item  For each $j \ge 0$, $\sup_{t\in[0,T]}\left|{\mathbb E}\left[ \Delta B_{\frac jn}B_t\right]\right| \le 2n^{-2H}$.
\item  For any $t \in [0,T]$ and integer $j\ge 1$, $\left|{\mathbb E}\left[ \Delta B_{\frac jn} B_t\right]\right| =\left|\left< \partial_{\frac jn}, \varepsilon_t\right>_\hten\right| \le Cn^{-2H}\left( j^{2H-1} + |j-nt|^{2H-1}\right)$.  In particular, if $|k-j| \ge 2$, $
\left|{\mathbb E}\left[ \Delta B_{\frac jn} \widetilde{B}_{\frac kn}\right]\right| = \left|\left< \partial_{\frac jn}, \widetilde{\varepsilon}_{\frac kn}\right>_\hten\right| \le n^{-2H}\left(j^{2H-1} + |j-k|^{2H-1}\right)$.
%\item  For any integer $j\ge 1$, $\left|{\mathbb E}\left[ \Delta B_{\frac jn} \widetilde{B}_{\frac jn}\right]\right| =\left|\left< \partial_{\frac jn}, \widetilde{\varepsilon}_{\frac jn}\right>_\hten\right| \le n^{-2H}j^{2H-1}$.
\end{enumerate}

As a result of properties (B.1) - (B.5), we have the following technical results.
\begin{lemma}Let $H < 1/2$ and $0 < t \le T$, and let $n \ge 2$ be an integer.  Then
\begin{enumerate}[(a)]
%\item  For any $0 \le j,k \le \lfloor nT \rfloor$ and $0 \le t \le T$, $\left|\left< \partial_{\frac jn}, \partial_{\frac kn}\right>_\hten\right| \le 2 \left|\left< \partial_{\frac jn}, \varepsilon_t\right>_\hten \right|$.
\item For fixed $0 \le s \le T$ and integer $r \ge 1$,
\[  \sum_{j=0}^{\Nt -1} \left| \left< \partial_{\frac jn}, \varepsilon_s\right>_\hten^r\right| \le Cn^{-2(r-1)H}.\]
\item For integer $r \ge 1$,
\[\sum_{j=0}^{\Nt -1} \left| \left< \partial_{\frac jn}, \widetilde \varepsilon_{\frac jn}\right>_\hten^r\right| \le Cn^{-2(r-1)H}.\]
\item For integers $r \ge 1$ and $0\le k \le \Nt$,
\[ \sum_{j=0}^{\Nt -1} \left|\left< \partial_{\frac jn}, \partial_{\frac kn}\right>_\hten^r\right| \le C n^{-2rH},\]
and consequently \[ \sum_{j,k=0}^{\Nt -1} \left|\left< \partial_{\frac jn}, \partial_{\frac kn}\right>_\hten^r\right| \le C\Nt n^{-2rH}.\]
\end{enumerate}
\end{lemma}
\begin{proof}
For (a), first note that  we have $\left|\left< \partial_0, \varepsilon_t\right>_\hten\right| \le T^Hn^{-H}$ by (B.1) and Cauchy-Schwarz.  Further, if $\left| \frac jn - s\right| < \frac 2n$, then by (B.4) we have $\left| \left< \partial_{\frac jn}, \widetilde \varepsilon_s\right>_\hten \right| \le Cn^{-2H}$.  Let ${\cal J} = \{ 1 \le j \le \Nt, |j - ns| > 1\}$; and note that $| {\cal J}^c| \le 2$.  Then for the case $r =1$ we have
\begin{align*}
\sum_{j=0}^{\Nt -1} \left| \left< \partial_{\frac jn}, \varepsilon_s\right>_\hten\right| &=
\left| \left< \partial_0, \varepsilon_t\right>_\hten\right| + \sum_{j\in{\cal J}^c} \left| \left< \partial_{\frac jn}, \varepsilon_s\right>_\hten\right| + \sum_{j\in {\cal J}} \left| \left< \partial_{\frac jn}, \varepsilon_s\right>_\hten\right|\\
&\le T^Hn^{-H} + Cn^{-2H} + Cn^{-2H}\sum_{j=1}^{\Nt -1} j^{2H-1} + | j -ns|^{2H-1}\\
&\le C \Nt^{2H} n^{-2H} \le C.
\end{align*}
For the case $r > 1$, we have by (B.4)
\[
\sum_{j=0}^{\Nt -1} \left| \left< \partial_{\frac jn}, \varepsilon_s\right>_\hten^r\right| \le \sup_{0\le j \le \Nt} \left| \left<\partial_{\frac jn}, \varepsilon_s\right>_\hten^{r-1}\right|
\sum_{j=0}^{\Nt -1} \left| \left< \partial_{\frac jn}, \varepsilon_s\right>_\hten\right|
\le Cn^{-2(r-1)H}.\]

For (b), we have by (B.4) and \req{cov_fbm}
\begin{align*}
\sum_{j=0}^{\Nt -1} \left| \left< \partial_{\frac jn}, \widetilde \varepsilon_{\frac jn}\right>_\hten^r\right| &\le \sup_{0\le j \le \Nt} \left| \left< \partial_{\frac jn}, \widetilde \varepsilon_{\frac jn}\right>_\hten^{r-1}\right|\sum_{j=0}^{\Nt -1} \left| \left< \partial_{\frac jn}, \widetilde \varepsilon_{\frac jn}\right>_\hten\right|\\
&\le Cn^{-2(r-1)H}\sum_{j=0}^{\Nt -1} \frac 12\left| {\mathbb E}\left[ \Delta B_{\frac jn}\left( B_{\frac jn} + B_{\frac{j+1}{n}}\right)\right]\right|\\
&= Cn^{-2(r-1)H}\sum_{j=0}^{\Nt -1} \frac 12\left| {\mathbb E}\left[B^2_{\frac{j+1}{n}} - B^2_{\frac jn}\right]\right|  \\
&= Cn^{-2(r-1)H}\sum_{j=0}^{\Nt -1} \frac 12\left[ \left(\frac{j+1}{n}\right)^{2H} - \left(\frac{j}{n}\right)^{2H}\right]\\
&\le Cn^{-2(r-1)H}\frac{\Nt}{n} \le Cn^{-2(r-1)H}.
\end{align*} 

For (c), we note that $\left|\left< \partial_{j/n}, \partial_0\right>_\hten\right| = \left|\left< \partial_{j/n}, \varepsilon_{ 1/n}\right>_\hten\right|\le n^{-2H}$. Also note that by (B.1) and Cauchy-Schwarz we have $\left|\left< \partial_{j/n}, \partial_{k/n}\right>_\hten\right| \le n^{-2H}$ for any $1 \le j,k\le \Nt$.   To begin the proof, we consider the case when $1 \le k\le \Nt -1$ is fixed.  Then
\begin{align*} \sum_{j=0}^{\Nt -1} \left| \left< \partial_{\frac jn}, \partial_{\frac kn}\right>_\hten^r\right|&\le \sup_{0\le\j\le\Nt}\left\{\sup_{0\le k\le\Nt} \left|\left< \partial_{\frac jn}, \partial_{\frac kn}\right>_\hten^{r-1}\right|\right\}\sum_{j=0}^{\Nt -1} \left|\left< \partial_{\frac jn}, \partial_{\frac kn}\right>_\hten\right|\\
&\le n^{-2(r-1)H}\left( n^{-2H} + \sum_{j=1}^{k-2}\left|\left< \partial_{\frac jn}, \partial_{\frac kn}\right>_\hten\right| + \sum_{j=k-1}^{k+1}\left|\left< \partial_{\frac jn}, \partial_{\frac kn}\right>_\hten\right| + \sum_{j=k+2}^{\Nt-1}\left|\left< \partial_{\frac jn}, \partial_{\frac kn}\right>_\hten\right|\right)\\
\end{align*}
Then we use (B.2) and (B.3) to write
\begin{align*}
n^{-2(r-1)H}&\left( n^{-2H} +\sum_{j=1}^{k-2}\left|\left< \partial_{\frac jn}, \partial_{\frac kn}\right>_\hten\right| + \sum_{j=k-1}^{k+1}\left|\left< \partial_{\frac jn}, \partial_{\frac kn}\right>_\hten\right| + \sum_{j=k+2}^{\Nt-1}\left|\left< \partial_{\frac jn}, \partial_{\frac kn}\right>_\hten\right|\right)\\
 &\quad\le n^{2(r-1)H}\left( n^{-2H} +Cn^{-2H}\sum_{j=1}^{k-2} (k-j)^{2H-2} + \sum_{j=k-1}^{k+1} n^{-2H} + Cn^{-2H}\sum_{j=k+2}^{\Nt -1} (j-k)^{2H-2}\right)\\
&\quad \le Cn^{-2rH}\left( 4 + 2\sum_{m=1}^\infty m^{2H-2}\right) \le Cn^{-2rH},
\end{align*}
where we note the sum is finite because $H < 1/2$.  For the double sum result we have
\[
 \sum_{j,k=0}^{\Nt -1} \left| \left< \partial_{\frac jn}, \partial_{\frac kn}\right>_\hten^r\right|
\le \sum_{k=0}^{\Nt -1} \sup_{0\le k\le \Nt}\left\{ \sum_{j=0}^{\Nt -1}\left|\left< \partial_{\frac jn}, \partial_{\frac kn}\right>_\hten^r\right|\right\}
\le C \Nt n^{-2rH}.
\]
%\begin{align*}\sum_{j,k=0}^{\Nt -1} \left|\left< \partial_{\frac jn}, \partial_{\frac kn}\right>_\hten^r\right| &=\sum_{j=0}^{\Nt %-1} \left<\partial_{\frac jn},\partial_{\frac jn}\right>_\hten^r + 2\sum_{j=1}^{\Nt -1}\left(\left|\left< \partial_{\frac jn}, %\partial_0 \right>_\hten^r\right|+\left|\left<\partial_{\frac jn}, \partial_{\frac{j-1}{n}}\right>_\hten^r\right|\right)
% + 2\sum_{j=1}^{\Nt -1}\sum_{k=1}^{j-2}\left|\left< \partial_{\frac jn}, \partial_{\frac kn}\right>_\hten^r\right|\\
%&\le 5\Nt n^{-2rH} + 2\sup_{0\le j \le \Nt}\left\{ \sup_{0\le k \le \Nt}\left|\left<\partial_{\frac jn},\partial_{\frac %kn}\right>_\hten^{r-1}\right|\right\}\sum_{j=1}^{\Nt}\sum_{k=1}^{j-2}\left|\left< \partial_{\frac jn}, \partial_{\frac %kn}\right>_\hten\right|\\
%&\le 5\Nt n^{-2rH} +2n^{-2rH}\sum_{j=1}^{\Nt }\sum_{k=1}^{j-2} (j-k)^{2H-2}\\
%&\le 5\Nt n^{-2rH} + 2n^{-2rH}\sum_{j=1}^{\Nt }\sum_{m=1}^\infty m^{2H-2} \le C\Nt n^{-2rH} .
%\end{align*}
\end{proof}

\section{Results}
\subsection{Some results for fBm with $H > 1/14$.} 
The following proposition summarizes some known results about stochastic integrals with respect to fBm, when the integrals arise from a Riemann sum construction.  A comprehensive treatment can be found in an important paper by Gradinaru, Nourdin, Russo \& Vallois \cite{GNRV}.
\begin{proposition}
Let $g \in {\cal C}^{\infty}({\mathbb R})$, such that $g$ and its derivatives have moderate growth.  The following Riemann sums converge in probability as $n \to \infty$ to $g(B_t) - g(0)$ for the given ranges of $H$:
\begin{enumerate}[(a)]
\item Midpoint rule:  for $1/6 < H < 1/2$,
\[\sum_{j=0}^{\Nt -1} g'(\widetilde B_{\frac jn})\Delta B_{\frac jn},\]
where $\widetilde B_{\frac jn} = \frac 12 \left(B_{\frac jn} + B_{\frac{j+1}{n}}\right)$.
\item Trapezoidal rule:  For $1/6 < H < 1/2$,
\[\sum_{j=0}^{\Nt -1} \frac 12 \left( g'(B_{\frac jn}) + g'(B_{\frac{j+1}{n}})\right)\Delta B_{\frac jn}.\]
\item Simpson's rule:  For $1/10 < H < 1/2$,
\[ \sum_{j=0}^{\Nt -1} \frac 16\left(g'(B_{\frac jn}) + 4 g'(\widetilde B_{\frac jn}) + g'(B_{\frac{j+1}{n}})\right)\Delta B_{\frac jn}.\]
\item Milne's rule:  For $1/14 < H < 1/2$,
\[ \sum_{j=0}^{\Nt - 1}\frac{1}{90}\left( 7g'(B_{\frac jn}) + 32g'(B_{\frac jn} +\frac 14 \Delta B_\frac jn) + 12g'(\widetilde B_{\frac{j}{n}}) + 32g'(B_{\frac{j}{n}}+\frac 34 \Delta B_{\frac jn}) + 7g'(B_{\frac{j+1}{n}})\right)\Delta B_{\frac jn}.\]
\end{enumerate}\end{proposition}
Note that the `midpoint' sum of part (a) is a different construction than that leading to \req{cov_14}.  All of these results follow from Theorem 4.4 of \cite{GNRV}, in fact they are also proved there for $H \ge 1/2$.  However, here we give a different proof of part (c).  By similar techniques, results (a), (b) and (d) could also be done in this way.  This proof will contain some results that will be used in Section 3.2, and help set up the proof of Theorem 3.3.  We begin with a technical result.  The proof of Lemma 3.2 is deferred to Section 4 due to length.

\begin{lemma} 
Let $r=1, 3, 5, \dots$ and $n\ge 2$ be an integer.  Let $\phi: {\mathbb R} \to {\mathbb R}$ be a ${\cal C}^{2r}$ function such that $\phi$ and all derivatives up to order $2r$ have moderate growth, and let $\{ B_t, t\ge 0\}$ be fBm with Hurst parameter $H$.  Then for each $r$, there is a constant $C >0$ such that
\[ {\mathbb E} \left[ \left( \sum_{j=0}^{\Nt -1} \phi(\Btilde) \Delta B_{\frac jn}^r\right)^2\right] \le C\sup_{0\le j \le \Nt}\left\| \phi(\Btilde)\right\|_{{\mathbb D}^{2r,2}}^2\Nt n^{-2rH},\]
where $C$ depends on $r$ and $H$.\end{lemma}

Now for the convergence of the Simpson's rule sum.  We begin with some elementary results from the calculus of deterministic functions.  For $x, h \in {\mathbb R}$ and a ${\cal C}^\infty$ function $g$, we have the following integral form for the Simpson's rule sum:
\begin{multline*}
g(x+h) - g(x-h) = \int_{-h}^h g'(x+u)~du\\
=  \frac h3 \left( g'(x-h) + 4g'(x) + g'(x+h)\right) + \frac 16 \int_0^h \left( g^{(4)}(x-u) - g^{(4)}(x+u)\right)~u(h-u)^2 du.\end{multline*}
See Talman \cite{Talman} for a nice discussion of the Simpson's rule error term.  Next, we consider a Taylor expansion of order 7 for $g^{(4)}$:
\begin{align*}
g^{(4)}(x+u) - g^{(4)}(x) &= \sum_{\ell =1}^6 \frac{g^{(4+\ell)}(x)}{\ell!}u^\ell \, + \, \frac{g^{(11)}(\xi)}{7!}u^7; \text{ and}\\
g^{(4)}(x) - g^{(4)}(x-u) &= \sum_{\ell =1}^6 \frac{(-1)^{\ell +1} g^{(4+\ell)}(x)}{\ell!}u^\ell \, + \, \frac{g^{(11)}(\eta)}{7!}u^7\end{align*}
Adding the above equations, we obtain
\[ g^{(4)}(x+u) - g^{(4)}(x-u) = 2\sum_{\nu = 1}^3 \frac{g^{(4+2\nu -1)}(x)}{(2\nu-1)!}u^{2\nu -1} ~+~\frac{g^{(11)}(\xi) + g^{(11)}(\eta)}{7!}u^7.\]
It follows that we can write
\begin{align*}
g(x+h) - g(x-h) &= \frac h3 \left( g'(x-h) + 4g'(x) + g'(x+h)\right) - \frac 13 \sum_{\nu = 1}^3 \frac{g^{(4+2\nu -1)}(x)}{(2\nu-1)!}\int_0^h u^{2\nu}(h-u)^2du\\
&\qquad\qquad - \frac{g^{(11)}(\xi)+g^{(11)}(\eta)}{(6)(7!)}\int_0^h u^8(h-u)^2 du\\
&= \frac h3 \left( g'(x-h) + 4g'(x) + g'(x+h)\right) - \frac{ g^{5)}(x)}{90}h^5 - A_7 g^{(7)}(x)h^7 - A_9 g^{(9)}(x)h^9\\
&\qquad\quad - \frac{1}{6(7!)}\int_0^h \left[ g^{(11)}(\xi) + g^{(11)}(\eta)\right]u^8( h-u)^2 du,\tag{*}\end{align*}
where $A_7, A_9$ are positive constants, and $\xi = \xi(u) \in [x-h, x+h]$, with similar for $\eta$.
With this relation, we now return to Proposition 3.1.c.  We begin with the telescoping series,
\begin{align*}
g(B_t) - g(0) &= \sum_{j=0}^{\Nt -1} \left( g(B_{\frac{j+1}{n}}) - g(B_{\frac jn})\right) + \left( g(B_t) - g(B_{\frac{\Nt}{n}})\right)\\
&= \sum_{j=0}^{\Nt -1} \int_{B_{j/n}}^{B_{(j+1)/n}} g'(u)~du +\left( g(B_t) - g(B_{\frac{\Nt}{n}})\right).\end{align*}
By continuity, the term $\left(g(B_t) - g(B_{\Nt / n})\right)$ tends to zero uniformly on compacts in probability (ucp) as $n \to \infty$, and may be neglected.  For each integral term, we use (*) with $x = \widetilde B_{j/n}$ and $h = \frac 12 \Delta B_{j/n}$ to obtain
\begin{multline}\label{simps_taylor}
 \sum_{j=0}^{\Nt -1} \int_{B_{j/n}}^{B_{(j+1)/n}} g'(u)~du =  \sum_{j=0}^{\Nt -1}\frac 16 \left( g'(B_{\frac jn}) + 4g'(\Btilde) + g'(B_{\frac{j+1}{n}})\right) - \frac{1}{2^5~90}\sum_{j=0}^{\Nt -1} g^{(5)}(\Btilde)\Delta B_{\frac jn}^5\\
 - A_7 \sum_{j=0}^{\Nt -1} g^{(7)}(\Btilde) \Delta B_{\frac jn}^7  -  A_9 \sum_{j=0}^{\Nt -1} g^{(9)}(\Btilde) \Delta B_{\frac jn}^9\\
 - \frac{1}{6(7!)} \sum_{j=0}^{\Nt -1} \int_0^{\Delta B_{j/n}} \left( g^{(11)}(\xi) + g^{(11)}(\eta)\right)u^8( \Delta B_{\frac jn}-u)^2 du.\end{multline}
By Lemma 3.2, the terms
\[ \sum_{j=0}^{\Nt -1} \frac{g^{(5)}(\Btilde)}{2880}\Delta B_{\frac jn}^5, \quad A_7 \sum_{j=0}^{\Nt -1} g^{(7)}(\Btilde) \Delta B_{\frac jn}^7, \quad   A_9 \sum_{j=0}^{\Nt -1} g^{(9)}(\Btilde) \Delta B_{\frac jn}^9 \]
all tend to zero in $L^2(\Omega)$ as $n \to \infty$.  For the last term, we have the $L^2(\Omega)$ estimate
\begin{multline*}
{\mathbb E}\left[\left( \sum_{j=0}^{\Nt -1} \int_0^{\Delta B_{j/n}} \left[ g^{(11)}(\xi) + g^{(11)}(\eta)\right]u^8( \Delta B_{\frac jn}-u)^2 du\right)^2\right]\\ \le C\left( {\mathbb E}\left[\sup_{s\in [0,t]}| g^{(11)}(B_s)^4|\right]\right)^{\frac 12} \left(\sum_{j=0}^{\Nt -1} \| \Delta B_{\frac jn}^{11}\|_{L^4(\Omega)}\right)^2
\le C\Nt^2 n^{-22H} \le Cn^{-2H},\end{multline*}
because $\| \Delta B_{j/n}^{11}\|_{L^4(\Omega)} \le C\left({\mathbb E} | \Delta B_{j/n}^2|\right)^{\frac{11}{2}} \le Cn^{-11H}$ by (B.1) and the Gaussian moments formula.  Thus, we have \[{\mathbb P}\lim_{n\to\infty} \sum_{j=0}^{\Nt -1} \frac 16 \left( g'(B_{\frac jn}) + 4g'(\Btilde)+g'(B_{\frac{j+1}{n}})\right)\Delta B_{\frac jn} = f(B_t) - f(0),\]
when $H > 1/10$, and Proposition 3.1.c is proved.$\qquad\square$

\medskip
As a converse to Proposition 3.1.c (and parts (a), (b) and (d) by similar computation), let $g(x) =f(x)$ be a polynomial such that $g^{(5)}=f^{(5)} = 1$.  Then
\[ S^S_n(t) = f(B_{\frac{\Nt}{n}}) - f(0) +\frac{1}{2880} \sum_{j=0}^{\Nt -1}\Delta B_{\frac jn}^5.\]
By Theorem 10 of Nualart and Ortiz-Latorre \cite{NuOrtiz}, the sequence $\left( B_t, \sum_{j=0}^{\Nt -1}\Delta B_{ j/n}^5\right)$ converges in distribution to $(B_t, W)$, where $W$ is a Gaussian random variable, independent of $B$.  It follows that $S^S_n(t)$ does not, in general, converge in probability when $H \le 1/10$.  For the critical case $H=1/10$, we have the following theorem, which generalizes the result of Theorem 10 of \cite{NuOrtiz} for this particular value of $H$.
%%%%%%%%%%%%%%%%%%%%%%%%%%%%%

\subsection{Main result: fBm with $H=1/10$.}%%%%%%%%%%%%%%%%
Throughout the rest of this paper, we will assume that $f: {\mathbb R}\to {\mathbb R}$ is a ${\cal C}^\infty$ function, such that $f$ and all derivatives satisfy moderate growth conditions. Note that this implies ${\mathbb E}\left[ \sup_{t\in[0,T]} \left| f^{(n)}(B_t)\right|^p\right] < \infty$ for all $n=0, 1,2, \dots$ and $1 \le p < \infty$.

\begin{theorem}%%%%%%%%Main Theorem 3.3
Let $f:{\mathbb R}\to{\mathbb R}$ be a ${\cal C}^\infty$ function such that $f$ and its derivatives have moderate growth conditions, and let $\{ B_t, t\ge 0\}$ be a fractional Brownian motion with $H = 1/10$.  For $t\ge 0$ and integers $n \ge 2$, Define
\[ S_n^S(t) = \sum_{j=0}^{\lfloor nt \rfloor -1} \frac 16\left(f'(B_{\frac{j}{n}}) + 4f'\left((B_{\frac jn}+B_{\frac{j+1}{n}})/2\right)+f'(B_{\frac{j+1}{n}})\right)\left(B_{\frac{j+1}{n}}-B_{\frac jn}\right).\]
Then as $n \to \infty$
\[ \left( B_t, S^S_n(t)\right) \law \left( B_t, f(B_t) - f(0) +\frac{\beta}{2^5\cdot90} \int_0^t f^{(5)}(B_s)~dW_s\right),\]
where $W = \{W_t, t\ge 0\}$ is a Brownian motion, independent of $B$, and 
\[ \beta = \sqrt{ (5!)2^{-5}\kappa_5 +75\kappa_3},\;\text{ for }\; \kappa_5 = \sum_{p\in{\mathbb Z}} \left( (p+1)^{\frac 15} - 2p^{\frac 15}+(p-1)^{\frac 15}\right)^5,\; \text{and}\]
\[\kappa_3 = \sum_{p\in{\mathbb Z}} \left( (p+1)^{\frac 15} - 2p^{\frac 15}+(p-1)^{\frac 15}\right)^3.\]
Consequently, 
\[ f(B_t) \stackrel{\cal L}{=} f(0) + \int_0^t f'(B_s)~d^SB_s - \frac{\beta}{2880} \int_0^t f^{(5)}(B_s)~dW_s,\]
where $\int_0^t f'(B_s)~d^SB_s$ denotes the weak limit of the `Simpson's rule' sum $S_n^S(t)$.
\end{theorem}

The rest of this section is given to proof of Theorem 3.3, and follows in Sections 3.3 - 3.5.  Following the telescoping series argument given in the proof of Proposition 3.1.c (see \req{simps_taylor}), we can write
\begin{multline*}
f(B_t) - f(0) = S^S_n(t) - \frac{1}{2^5~90}\sum_{j=0}^{\Nt -1} f^{(5)}(\Btilde)\Delta B_{\frac jn}^5
 - A_7 \sum_{j=0}^{\Nt -1} f^{(7)}(\Btilde) \Delta B_{\frac jn}^7  -  A_9 \sum_{j=0}^{\Nt -1} f^{(9)}(\Btilde) \Delta B_{\frac jn}^9\\
 - \frac{1}{6(7!)} \sum_{j=0}^{\Nt -1} \int_0^{\Delta B_{j/n}} \left( f^{(11)}(\xi) + f^{(11)}(\eta)\right)u^8( \Delta B_{\frac jn}-u)^2 du +\left( f(B_t) - f(B_{\frac{\Nt}{n}})\right).\end{multline*}
As in the proof of Proposition 3.1.c, for $H = 1/10$ it follows from Lemma 3.2 that the terms including $A_7,~A_9$ and the integral term all tend to zero in $L^2(\Omega)$ as $n\to\infty$, and the term $\left(f(B_t) - f(B_{\Nt /n})\right)$ also tends to zero ucp as $n\to\infty$.  The main task to prove Theorem 3.3, then, is to show convergence in law of the error term
\beq{simps_err} \sum_{j=0}^{\lfloor nt \rfloor -1} f^{(5)}(\Btilde) \Delta B_{\frac jn}^5.\eeq

\medskip
\subsection{Malliavin calculus representation.}%%%%Step 2
In order to apply our convergence theorem (Theorem 2.3), we wish to find a Malliavin calculus representation for the term \req{simps_err}.  Consider the Hermite polynomial identity
$H_5 (x)  = x^5 -10H_3(x) - 15x.$ 
Taking $x = \Delta B_{j/n}/\| \Delta B_{j/n}\|_{L^2(\Omega)} = n^{H}\Delta B_{j/n}$, we have
\[ n^{5H}\Delta B_{\frac jn}^5 = H_5 (n^{H} \Delta B_{\frac jn}) + 10H_3(n^{H}\Delta B_{\frac jn}) + 15n^{H}\Delta B_{\frac jn}.\]Using \req{Hmap}, this gives
\begin{multline*} \sum_{j=0}^{\Nt -1}f^{(5)}(\widetilde B_{\frac jn}) \Delta B_{\frac jn}^5 = \sum_{j=0}^{\Nt -1}f^{(5)}(\widetilde B_{\frac jn})\delta^5 (\partial_{\frac jn}^{\otimes 5})\\+ 10n^{-2H}\sum_{j=0}^{\Nt -1}f^{(5)}(\widetilde B_{\frac jn})\delta^3(\partial_{\frac jn}^{\otimes 3}) + 15n^{-4H}\sum_{j=0}^{\Nt -1}f^{(5)}(\widetilde B_{\frac jn})\Delta B_{\frac jn}.\end{multline*}
We first show that the last term tends to zero in $L^1(\Omega)$.
\begin{lemma}  Under the assumptions of Theorem 3.3, there is a constant $C>0$ such that
\[{\mathbb E} \left[\left(  n^{-4H} \sum_{j=0}^{\Nt -1} f^{(5)}(\Btilde) \Delta B_{\frac jn} \right)^2\right] \le Cn^{-2H}.\]
\end{lemma}
\begin{proof}
We start with a 2-sided Taylor expansion of $f^{(4)}$ of order 7.  That is,
\[f^{(4)}(B_{\frac{j+1}{n}}) - f^{(4)}( \widetilde B_{\frac jn}) = \sum_{\ell=1}^6 \frac{f^{(4+\ell)}( \widetilde B_{\frac jn})}{2^\ell \ell!}\Delta B_{\frac jn}^\ell +  \frac{f^{(11)}(\xi_j)}{2^7 7!}\Delta B_{\frac jn}^7\]
and
\[f^{(4)}(\widetilde B_{\frac jn})- f^{(4)}(B_{\frac jn}) = \sum_{\ell=1}^6 \frac{(-1)^{\ell+1}f^{(4+\ell)}( \widetilde B_{\frac jn})}{2^\ell \ell!}\Delta B_{\frac jn}^\ell +  \frac{f^{(11)}(\eta_j)}{2^7 7!}\Delta B_{\frac jn}^7,\]
for some intermediate values $\xi_j, \eta_j$ between $B_{j/n}$ and $B_{(j+1)/n}$.  Adding the above equations, we obtain
\begin{multline}\label{Taylor_3}f^{(4)}(B_{\frac{j+1}{n}}) - f^{(4)}(B_{\frac jn}) = f^{(5)}(\widetilde B_{\frac jn})\Delta B_{\frac jn} +\frac{f^{(7)}(\Btilde)}{24}\Delta B_{\frac jn}^3 + \frac{f^{(9)}(\Btilde)}{2^4 5!}\Delta B_{\frac jn}^5\\
+ \frac{f^{(11)}(\xi_j)+f^{(11)}(\eta_j)}{2^7 7!}\Delta B_{\frac jn}^7.\end{multline}
It follows that we can write
\begin{multline*}  {\mathbb E} \left[\left(  n^{-4H} \sum_{j=0}^{\Nt -1} f^{(5)}(\Btilde) \Delta B_{\frac jn} \right)^2\right]  \le 4{\mathbb E}\left[\left( n^{-4H}\sum_{j=0}^{\Nt -1} \left(f^{(4)}(B_{\frac{j+1}{n}}) - f^{(4)}(B_{\frac jn})\right)\right)^2\right]\\
+4{\mathbb E}\left[\left(n^{-4H}\sum_{j=0}^{\Nt -1} \frac{ f^{(7)}(\Btilde)}{24}\Delta B_{\frac jn}^3\right)^2\right]
+4{\mathbb E}\left[\left(n^{-4H}\sum_{j=0}^{\Nt -1} \frac{ f^{(9)}(\Btilde)}{2^4 5!}\Delta B_{\frac jn}^5\right)^2\right]\\
+4{\mathbb E}\left[\left(n^{-4H}\sum_{j=0}^{\Nt -1}\frac{f^{(11)}(\xi_j) + f^{(11)}(\eta_j)}{2^7 7!}\Delta B_{\frac jn}^7\right)^2\right]
.\end{multline*}
By growth assumptions on $f^{(4)}$,
\[ {\mathbb E}\left[\left( n^{-4H}\sum_{j=0}^{\Nt -1} \left(f^{(4)}(B_{\frac{j+1}{n}}) - f^{(4)}(B_{\frac jn})\right)\right)^2\right] = n^{-8H} {\mathbb E}\left[\left( f^{(4)}(B_{\frac{\Nt}{n}}) - f^{(4)}(0)\right)^2\right] \le Cn^{-8H}.\]
By Lemma 3.2,
\[{\mathbb E}\left[\left(n^{-4H}\sum_{j=0}^{\Nt -1} \frac{ f^{(7)}(\Btilde)}{24}\Delta B_{\frac jn}^3\right)^2\right] \le C\sup_{0\le j \le \Nt} \|f^{(7)}(\Btilde)\|^2_{{\mathbb D}^{6,2}}\Nt n^{-14H},\]
and 
\[{\mathbb E}\left[\left(n^{-4H}\sum_{j=0}^{\Nt -1} \frac{ f^{(9)}(\Btilde)}{2^4 5!}\Delta B_{\frac jn}^5\right)^2\right] \le C\sup_{0\le j \le \Nt} \|f^{(9)}(\Btilde)\|^2_{{\mathbb D}^{10,2}}\Nt n^{-18H}.\]
Then by (B.1),
\begin{multline*}
{\mathbb E}\left[\left(n^{-4H}\sum_{j=0}^{\Nt -1}\frac{f^{(11)}(\xi_j) + f^{(11)}(\eta_j)}{2^7 7!}\Delta B_{\frac jn}^7\right)^2\right]\\ \le C\left( {\mathbb E}\left[\sup_{s\in[0,t]} \left|f^{(11)}(B_s)^4\right|\right]\right)^{\frac 12}n^{-8H}\left(\sum_{j=0}^{\Nt -1} \|\Delta B_{\frac jn}^7\|_{L^4(\Omega)}\right)^2
\le C\Nt^2 n^{-22H} \le Cn^{-2H}.\end{multline*}
This proves the lemma.\end{proof}

Lemma 3.4 shows that only the terms
\[\sum_{j=0}^{\Nt -1} f^{(5)}(\Btilde) \delta^5\left( \partial_{\frac jn}^{\otimes 5}\right) + 10n^{-2H}\sum_{j=0}^{\Nt -1} f^{(5)}(\Btilde) \delta^3(\partial_{\frac jn}^{\otimes 3})\]
are significant.  Using Lemma 2.1.a, we can write the first term as
\[ \sum_{j=0}^{\Nt -1} \delta^5\left( f^{(5)}(\Btilde)\partial_{\frac jn}^{\otimes 5}\right) + \sum_{r=1}^5\binom 5r\sum_{j=0}^{\Nt -1} \delta^{5-r}\left( f^{(5+r)}(\Btilde)\partial_{\frac jn}^{\otimes(5-r)}\right)\left< \widetilde{\varepsilon}_{\frac jn}, \partial_{\frac jn}\right>_\hten^r.\]

By Lemma 2.1.c and (B.1), we have the estimate
\[ \left\| \delta^{(5-r)}\left( f^{(5+r)}(\widetilde B_{\frac jn})\partial_{\frac jn}^{\otimes (5-r)}\right)\right\|_{L^2(\Omega)} \le C\left\| \partial_{\frac jn}^{\otimes (5-r)}\right\|_{\hten^{\otimes 5-r}} \le Cn^{(r-5)H}.\]
It follows that for $r= 1, \dots, 5$, we can use Lemma 2.6.b,
\begin{multline*} {\mathbb E}\left| \binom 5r \sum_{j=0}^{\Nt -1} \delta^{(5-r)}\left( f^{(5+r)}(\widetilde B_{\frac jn})\partial_{\frac jn}^{\otimes (5-r)}\right) \left< \widetilde \varepsilon_{\frac jn}, \partial_{\frac jn}\right>_\hten^r\right|\\ 
\le Cn^{(r-5)H} \sum_{j=0}^{\Nt -1} \left| \left<\widetilde \varepsilon_{\frac jn}, \partial_{\frac jn}\right>_\hten^r\right| \le Cn^{-(3+r)H}.
\end{multline*}
By a similar computation, 
\begin{multline*} 10n^{-2H} \sum_{j=0}^{\Nt -1} f^{(5)}(\Btilde) \delta^3(\partial_{\frac jn}^{\otimes 3}) = 10n^{-2H}\sum_{j=0}^{\Nt -1} \delta^3 \left( f^{(5)}(\Btilde) \partial_{\frac jn}^{\otimes 3}\right)\\
+ 10n^{-2H}\sum_{r=1}^3\binom 3r \sum_{j=0}^{\Nt -1} \delta^{(3-r)}\left( f^{(5+r)}(\widetilde B_{\frac jn})\partial_{\frac jn}^{\otimes (3-r)}\right) \left< \widetilde \varepsilon_{\frac jn}, \partial_{\frac jn}\right>_\hten^r,
\end{multline*}
where
\[n^{-2H}{\mathbb E}\left|\sum_{r=1}^3\binom 3r \sum_{j=0}^{\Nt -1} \delta^{(3-r)}\left( f^{(5+r)}(\widetilde B_{\frac jn})\partial_{\frac jn}^{\otimes (3-r)}\right) \left< \widetilde \varepsilon_{\frac jn}, \partial_{\frac jn}\right>_\hten^r\right| \le Cn^{-4H}.\]
Therefore, we define
\begin{align*} F_n &:= \sum_{j=0}^{\Nt -1} \delta^5\left( f^{(5)}(\Btilde)\partial_{\frac jn}^{\otimes 5}\right) = \delta^5(u_n),\;\;\text{ where }\; u_n = \sum_{j=0}^{\Nt -1}f^{(5)}(\Btilde)\partial_{\frac jn}^{\otimes 5}; \,\text{and}\\ 
G_n &:= 10n^{-2H}\sum_{j=0}^{\Nt -1} \delta^3\left( f^{(5)}(\Btilde)\partial_{\frac jn}^{\otimes 5}\right) = \delta^3(v_n),\;\;\text{ where }\; v_n = 10n^{-2H}\sum_{j=0}^{\Nt -1}f^{(5)}(\Btilde)\partial_{\frac jn}^{\otimes 3}.\end{align*}
 It follows that for large $n$, the term \req{simps_err} may be represented as $F_n + G_n + \epsilon_n$, where $\epsilon_n \to 0$ in $L^1(\Omega)$.  Then, as introduced in Remark 2.5, we will work with the vector sequence $(F_n, G_n)$.

\medskip
\subsection{Conditions of Theorem 2.3.}  %%%%%Step 3
Our main task in this step is to show that the sequence of random vectors $(F_n, G_n)$ satisfies the conditions of Theorem 2.3.  The first condition is that $(F_n, G_n)$ is bounded in $L^1(\Omega)$.  In fact, we have a stronger result that will also be helpful with later conditions.

\begin{lemma} Fix real numbers $0 < t \le T$ and $p \ge 2$, and integer $n \ge 2$. Let $\phi : {\mathbb R} \to {\mathbb R}$ be a ${\cal C}^\infty$ function such that $\phi$ and all its derivatives have moderate growth.  For integer $1 \le q \le 5$, define
\[ w_n = \sum_{j=0}^{\Nt -1} \phi (\Btilde) \partial_{\frac jn}^{\otimes q}.\]
Then for integers $0 \le a \le 5$, there exists a constant $c_{q,a}$ such that
\[ \left\| D^a \delta^q(w_n) \right\|^2_{L^p(\Omega; \hten^{\otimes a})} \le c_{q,a} \sup_{0\le j \le \Nt} \left\| \phi (\Btilde)\right\|_{{\mathbb D}^{q+a,p}}^2 \Nt n^{-2qH}\le Cn^{1-2qH}.\]
In particular, 
\beq{FGMeyer} \left\| D^a F_n\right\|_{L^p(\Omega; \hten^{\otimes a})} + \left\| D^a G_n\right\|_{L^p(\Omega; \hten^{\otimes a})} \le C.\eeq
\end{lemma}
\begin{proof}
This proof follows a similar result in \cite{NoNu}, see Theorem 5.2.  First, note that by Lemma 2.6.c and growth conditions on $\phi$, for each integer $b\ge 0$,
\begin{align*} \left\| D^b w_n\right\|^2_{\hten^{\otimes q+b}} &= \left\| \sum_{j=0}^{\Nt -1} \phi^{(b)}(\Btilde)\partial_{\frac jn}^{\otimes q}\otimes \widetilde \varepsilon_{\frac jn}^{\otimes b}\right\|_{\hten^{\otimes q+b}}^2\\
&\le \sup_{0\le j \le \Nt} \left| \phi^{(b)}(\Btilde)\right|^2~\sup_{0\le j,k \le \Nt} \left|\left< \widetilde \varepsilon_{\frac jn},  \widetilde \varepsilon_{\frac kn}\right>^b\right| \sum_{j,k=0}^{\Nt -1} \left|\left< \partial_{\frac jn}, \partial_{\frac kn}\right>_\hten^q\right|\\
&\le C\Nt n^{-2qH}\sup_{0\le j \le \Nt} \left| \phi^{(b)}(\Btilde)\right|^2 .\end{align*}
It follows that for $p \ge 2$,
\[ \left\| D^b w_n\right\|_{L^p(\Omega; \hten^{\otimes q+b})}^2 \le C \Nt n^{-2qH}{\mathbb E}\left[ \sup_{0\le j \le \Nt} \left| \phi^{(b)}(\Btilde)\right|^p\right]^{\frac 2p}.\]
Then, using the Meyer inequality (see \cite{NoNu}, Proposition 1.5.7),
\beq{Meyer}
\left\| D^a\delta^q(w_n) \right\|_{L^p(\Omega; \hten^{\otimes a})}^2\le \left\| \delta^q (w_n) \right\|^2_{{\mathbb D}^{a,p}} \le C \Nt n^{-2qH}\sup_{0\le j \le \Nt} \left\| \phi(\Btilde) \right\|_{{\mathbb D}^{q+a,p}(\hten^q)}^2\le C\Nt n^{-2qH}.\eeq
For \req{FGMeyer}, we have
\[ \left\| D^a F_n\right\|^2_{L^p(\Omega; \hten^{\otimes a})} = \left\| D^a \delta^5(u_n)\right\|^2_{L^p(\Omega; \hten^{\otimes a})}\le C\Nt n^{-10H}\sup_{0\le j \le \Nt} \left\| f^{(5)}(\Btilde)\right\|^2_{{\mathbb D}^{5+a,p}(\hten^{\otimes 5})} \le C,\]
and
\[ \left\| D^a G_n\right\|^2_{L^p(\Omega; \hten^{\otimes a})} = \left\| n^{-2H}D^a \delta^3(u_n)\right\|^2_{L^p(\Omega; \hten^{\otimes a})}\le C\Nt n^{-10H}\sup_{0\le j \le \Nt} \left\| f^{(5)}(\Btilde)\right\|^2_{{\mathbb D}^{3+a,p}(\hten^{\otimes 3})} \le C.\]

\end{proof}
The fact that $(F_n, G_n)$ is bounded in $L^1(\Omega)$ follows by taking $a=0$.  Next, we consider condition (a) of Theorem 2.3.
\begin{lemma}
Under the assumptions of Theorem 3.3, $(F_n, G_n)$ satisfies condition (a) of Theorem 2.3.  That is, we have
\begin{enumerate}[(a)]
\item For arbitrary $h \in\hten^{\otimes 5}$ and $g\in \hten^{\otimes 3}$,\[\lim_{n\to \infty} {\mathbb E}\left|\left< u_n, h\right>_{\hten^{\otimes 5}}\right| =\lim_{n\to\infty}{\mathbb E}\left|\left< v_n, g\right>_{\hten^{\otimes 3}}\right| =0.\] 
\item $\lim_{n\to\infty}{\mathbb E}\left|\left< u_n, \bigotimes_{i=1}^s D^{a_i} F_n \bigotimes_{i=s+1}^r D^{a_i}G_n \otimes h\right>_{\hten^{\otimes 5}}\right|=0$, where $0 \le a_i < 5$, $1\le a_1 + \cdots + a_r < 5$, and $h \in \hten^{\otimes 5-(a_1+\cdots +a_r)}$; and $\lim_{n\to\infty}{\mathbb E}\left|\left< v_n, \bigotimes_{i=1}^s D^{b_i} F_n \bigotimes_{i=s+1}^r D^{b_i}G_n \otimes g\right>_{\hten^{\otimes 3}}\right|=0$, where $0 \le b_i < 3$,  $1\le b_1 + \cdots + b_r < 3$, and $g \in \hten^{\otimes 3 - (b_1 +\cdots + b_r)}$.
\item $\lim_{n\to\infty}{\mathbb E}\left|\left< u_n, \bigotimes_{i=1}^s D^{a_i} F_n \bigotimes_{i=s+1}^r D^{a_i}G_n \right>_{\hten^{\otimes 5}}\right|=0$, where $r \ge 2$, $0 \le a_i < 5$ and $a_1 + \cdots + a_r =5$; and

$\lim_{n\to\infty}{\mathbb E}\left|\left< v_n, \bigotimes_{i=1}^s D^{b_i} F_n \bigotimes_{i=s+1}^r D^{b_i}G_n \right>_{\hten^{\otimes 3}}\right|=0$, where $r \ge 2$, $0 \le b_i < 3$ and $b_1 + \cdots + b_r =3$.
\end{enumerate}
\end{lemma}

The proof of this lemma is deferred to Section 4 due to its length.  To verify condition (b) of Theorem 2.3, we have four terms to consider:
\begin{itemize}
\item $\left< u_n, D^5 G_n \right>_{\hten^{\otimes 5}}$
\item $\left< v_n, D^3 F_n \right>_{\hten^{\otimes 3}}$
\item $\left< u_n, D^5 F_n \right>_{\hten^{\otimes 5}}$
\item $\left< v_n, D^3 G_n \right>_{\hten^{\otimes 3}}$
\end{itemize}

We deal with the first two terms in the following lemma.  The proof is given in Section 4.
\begin{lemma}Under the assumptions of Theorem 3.3, we have
\begin{enumerate}[(a)]
\item $\lim_{n \to \infty} {\mathbb E} \left| \left< u_n, D^5 G_n \right>_{\hten^{\otimes 5}}\right| = 0$
\item $\lim_{n \to \infty} {\mathbb E} \left| \left< v_n, D^3 F_n \right>_{\hten^{\otimes 3}}\right| = 0$.
\end{enumerate}\end{lemma}

\medskip
This leaves the variance terms.  Lemma 2.1.b allows us to write
\begin{align*}
\left< u_n, D^5F_n\right>_{\hten^{\otimes 5}} &= \sum_{j,k = 0}^{\Nt -1} \left< f^{(5)}(\Btilde) \partial_{\frac jn}^{\otimes 5}, D^5\delta^5\left( f^{(5)}(\widetilde B_{\frac kn})\partial_{\frac kn}^{\otimes 5}\right)\right>_{\hten^{\otimes 5}}\\
&= \sum_{z=0}^4 {\binom 5z}^2 z! \sum_{j,k = 0}^{\Nt -1} \left< f^{(5)}(\Btilde) \partial_{\frac jn}^{\otimes 5},\delta^{5-z}\left( f^{(10-z)}(\widetilde B_{\frac kn})\partial_{\frac kn}^{\otimes 5-z}\right)\partial_{\frac kn}^{\otimes z}\otimes \widetilde \varepsilon_{\frac kn}^{\otimes 5-z}\right>_{\hten^{\otimes 5}}\\
&\qquad\;\; + 5! \sum_{j,k = 0}^{\Nt -1} \left< f^{(5)}(\Btilde) \partial_{\frac jn}^{\otimes 5}, f^{(5)}(\widetilde B_{\frac kn})\partial_{\frac kn}^{\otimes 5}\right>_{\hten^{\otimes 5}}.\end{align*}
We first deal with the case $0 \le z \le 4$.  We have
\begin{multline*}
{\mathbb E} \sum_{j,k=0}^{\Nt -1} \left| \left<f^{(5)}(\Btilde) \partial_{\frac jn}^{\otimes 5},\delta^{5-z}\left( f^{(10-z)}(\widetilde B_{\frac kn})\partial_{\frac kn}^{\otimes 5-z}\right)\partial_{\frac kn}^{\otimes z}\otimes \widetilde \varepsilon_{\frac kn}^{\otimes 5-z}\right>_{\hten^{\otimes 5}}\right|\\
\le C \sup_{0\le j \le \Nt} \left\| f^{(5)}(\Btilde)\right\|_{L^2(\Omega)} \sup_{0\le k \le \Nt}\left\| \delta^{5-z}\left(f^{(10-z)}(\widetilde B_{\frac kn})\partial_{\frac kn}^{\otimes 5-z}\right)\right\|_{L^2(\Omega)}\\
\times  \sum_{j,k=0}^{\Nt -1} \left| \left< \partial_{\frac jn}, \partial_{\frac kn}\right>_\hten^z \left< \partial_{\frac jn}, \widetilde \varepsilon_{\frac kn}\right>_\hten^{5-z}\right|.
\end{multline*}
By (B.1) and Lemma 2.1.c, we have
\[\sup_{0\le k \le \Nt}\left\| \delta^{5-z}\left(f^{(10-z)}(\widetilde B_{\frac kn})\partial_{\frac kn}^{\otimes 5-z}\right)\right\|_{L^2(\Omega)} \le C\| \partial_{\frac 1n} \|^{5-z}_\hten \le Cn^{-(5-z)H},\]
so for the case $z=0$, we have
\begin{multline*}\sup_{0\le j \le \Nt} \left\| f^{(5)}(\Btilde)\right\|_{L^2(\Omega)} \sup_{0\le k \le \Nt}\left\| \delta^{5-z}\left(f^{(10-z)}(\widetilde B_{\frac kn})\partial_{\frac kn}^{\otimes 5-z}\right)\right\|_{L^2(\Omega)}\\
\times  \sum_{j,k=0}^{\Nt -1} \left| \left< \partial_{\frac jn}, \partial_{\frac kn}\right>_\hten^z \left< \partial_{\frac jn}, \widetilde \varepsilon_{\frac kn}\right>_\hten^{5-z}\right|\\
\le Cn^{-5H} \sup_{0\le j \le \Nt} \left\{\sup_{s\in[0,t]} \left| \left< \partial_{\frac jn}, \varepsilon_s\right>_\hten^4\right|\right\} \sum_{k=0}^{\Nt -1} \sup_{0\le k \le \Nt} \sum_{j=0}^{\Nt -1} \left|\left< \partial_{\frac jn}, \widetilde \varepsilon_{\frac kn}\right>_\hten\right|.\end{multline*}
By (B.4) and Lemma 2.6.a, respectively, 
\[\sup_{0\le j \le \Nt} \left\{\sup_{s\in[0,t]} \left| \left< \partial_{\frac jn}, \varepsilon_s\right>_\hten^4\right|\right\} \le Cn^{-8H}\;\text{ and} \;  \sup_{0\le k \le \Nt} \sum_{j=0}^{\Nt -1} \left|\left< \partial_{\frac jn}, \widetilde \varepsilon_{\frac kn}\right>_\hten\right| \le C,\]
so this gives
\[ Cn^{-5H} \sup_{0\le j \le \Nt} \left\{\sup_{s\in[0,t]} \left| \left< \partial_{\frac jn}, \varepsilon_s\right>_\hten^4\right|\right\} \sum_{k=0}^{\Nt -1} \sup_{0\le k \le \Nt} \sum_{j=0}^{\Nt -1} \left|\left< \partial_{\frac jn}, \widetilde \varepsilon_{\frac kn}\right>_\hten\right| \le C\Nt n^{-13H} \le Cn^{-3H}.\]
 If $1 \le z \le 4$, then by (B.1), (B.4) and Lemma 2.6.c we have an upper bound of 
\begin{multline*}
\sup_{0\le j \le \Nt} \left\| f^{(5)}(\Btilde)\right\|_{L^2(\Omega)} \sup_{0\le k \le \Nt}\left\| \delta^{5-z}\left(f^{(10-z)}(\widetilde B_{\frac kn})\partial_{\frac kn}^{\otimes 5-z}\right)\right\|_{L^2(\Omega)}
\sum_{j,k=0}^{\Nt -1} \left| \left< \partial_{\frac jn}, \partial_{\frac kn}\right>_\hten^z \left< \partial_{\frac jn}, \widetilde \varepsilon_{\frac kn}\right>_\hten^{5-z}\right|\\
\le C \| \partial_{\frac 1n}\|_\hten^{5-z}~\sup_{0\le j\le \Nt}\left\{ \sup_{s\in[0,t]} \left| \left< \partial_{\frac jn}, \varepsilon_s\right>^{5-z}\right|\right\}\sum_{j,k=0}^{\Nt -1} \left|\left< \partial_{\frac jn}, \partial_{\frac kn}\right>_\hten^z\right|
\le C\Nt n^{-(15-z)H} \le Cn^{-H},\end{multline*}
because $z< 5$.  It follows that the term corresponding to each $z = 0, \dots, 4$ vanishes in $L^1(\Omega)$, and we have that only the term with $z = 5$ is significant.  For the case $z=5$, we use a result from \cite{NoNu}, see proof of Theorem 5.2.
\begin{align*}
5! \sum_{j,k = 0}^{\Nt -1} & \left< f^{(5)}(\Btilde) \partial_{\frac jn}^{\otimes 5}, f^{(5)}(\widetilde B_{\frac kn})\partial_{\frac kn}^{\otimes 5}\right>_{\hten^{\otimes 5}}\\
&\quad = 5! \sum_{j,k=0}^{\Nt -1} f^{(5)}(\Btilde)  f^{(5)}(\widetilde B_{\frac kn}) \left({\mathbb E}\left[ \Delta B_{\frac jn}, \Delta B_{\frac kn}\right]\right)^5\\
&\quad = \frac{5!}{2^5n^{10H}} \sum_{p=-\infty}^\infty \sum_{j=(0 \vee -p)}^{(\Nt -1) \wedge (\Nt -1-p)} f^{(5)}(\Btilde) f^{(5)}(\widetilde B_{\frac{j+p}{n}}) \left( |p+1|^{2H} - 2|p|^{2H} + |p-1|^{2H}\right)^5,\end{align*}
which (for $H = 1/10$) converges in $L^1 (\Omega)$ to
\beq{Var_F_n} \frac{5!}{2^5} \kappa_5 \int_0^t f^{(5)}(B_s)^2~ds, \; \text{where } \kappa_5 = \sum_{p\in {\mathbb Z}} \left( |p+1|^{\frac 15} - 2|p|^{\frac 15} + |p-1|^{\frac 15}\right)^5.\eeq  Hence, we have that
\beq{var_F_final} \lim_{n \to \infty} \left< u_n, D^5 F_n \right>_{\hten^{\otimes 5}} = \frac{5!}{2^5} \kappa_5 \int_0^t f^{(5)} (B_s)^2 ds.\eeq
Similarly, we have
\[ \left< v_n, D^3 G_n\right>_{\hten^{\otimes 3}} = 10^2n^{-4H}\sum_{z=0}^3{\binom 3z}^2z!\sum_{j,k=0}^{\Nt -1} \left< f^{(5)}(\Btilde)\partial_{\frac jn}^{\otimes 3}, \delta^{3-z}\left(f^{(8-z)}(\widetilde B_{\frac kn})\partial_{\frac kn}^{\otimes 3-z}\right)\partial_{\frac kn}^{\otimes z}\otimes \widetilde \varepsilon_{\frac kn}^{\otimes 3-z}\right>_{\hten^{\otimes 3}}.\]
For $z=0$,
\begin{align*}
&100n^{-4H} {\mathbb E}\left| \sum_{j,k=0}^{\Nt -1} \left< f^{(5)}(\Btilde)\partial_{\frac jn}^{\otimes 3}, \delta^3\left(f^{(8)}(\widetilde B_{\frac kn})\partial_{\frac kn}^{\otimes 3}\right) \widetilde \varepsilon_{\frac kn}^{\otimes 3}\right>_{\hten^{\otimes 3}}\right|\\
&\qquad \le 100n^{-4H} \sup_{0\le j \le \Nt} \left\| f^{(5)}(\Btilde)\right\|_{L^2(\Omega)}\sup_{0\le k \le \Nt}\left\|\delta^3\left(f^{(8)}(\widetilde B_{\frac kn})\partial_{\frac kn}^{\otimes 3}\right)\right\|_{L^2(\Omega)} \sup_{j,k}\left| \left<\partial_{\frac jn}, \widetilde \varepsilon_{\frac kn}\right>_\hten^2\right|\\
&\qquad\qquad\;\;\times \sum_{k=0}^{\Nt-1} \sup_{s\in[0,t]}\sum_{j=0}^{\Nt -1} \left|\left< \partial_{\frac jn}, \varepsilon_s\right>_\hten \right|\\
&\qquad \le C\Nt n^{-11H} \le Cn^{-H}.\end{align*}
For $z=1$ or $z=2$, by (B.4) and Lemma 2.6.c,
\begin{align*}
&100{\binom 3z}^2 z! n^{-4H} {\mathbb E}\left| \sum_{j,k=0}^{\Nt -1}\left< f^{(5)}(\Btilde)\partial_{\frac jn}^{\otimes 3}, \delta^{3-z}\left(f^{(8-z)}(\widetilde B_{\frac kn})\partial_{\frac kn}^{\otimes 3-z}\right)\partial_{\frac kn}^{\otimes z}\otimes \widetilde \varepsilon_{\frac kn}^{\otimes 3-z}\right>_{\hten^{\otimes 3}}\right|\\
&\qquad \le Cn^{-4H}\sup_{0\le j \le \Nt} \left\| f^{(5)}(\Btilde)\right\|_{L^2(\Omega)}\sup_{0\le k \le \Nt}\left\|\delta^{3-z}\left(f^{(8-z)}(\widetilde B_{\frac kn})\partial_{\frac kn}^{\otimes 3-z}\right)\right\|_{L^2(\Omega)} \sup_{j,k}\left| \left<\partial_{\frac jn}, \widetilde \varepsilon_{\frac kn}\right>_\hten^{3-z}\right|\\
&\qquad\qquad\;\;\times\sum_{j,k=0}^{\Nt -1}\left| \left<\partial_{\frac jn}, \partial_{\frac kn}\right>_\hten^z\right|\\
&\qquad \le C\Nt n^{-(13-z)H} \le Cn^{-H},\end{align*}
because $z \le 2$.  Then for $z=3$, we have
\begin{multline*}
600n^{-4H}\sum_{j,k=0}^{\Nt -1}\left< f^{(5)}(\Btilde)\partial_{\frac jn}^{\otimes 3}, f^{(5)}(\widetilde B_{\frac kn})\partial_{\frac kn}^{\otimes 3}\right>_{\hten^{\otimes 3}}\\
= \frac{600}{2^3n^{10H}}\sum_{j,k=0}^{\Nt -1}f^{(5)}(\Btilde)f^{(5)}(\widetilde B_{\frac kn})\left( |j-k+1|^{2H}-2|j-k|^{2H}+|j-k-1|^{2H}\right)^3.
\end{multline*}
Similar to \req{Var_F_n}, this converges in $L^1(\Omega)$ to
\beq{Var_G_n} 75 \kappa_3 \int_0^t f^{(5)}(B_s)^2~ds, \; \text{where } \kappa_3 = \sum_{p\in {\mathbb Z}} \left( |p+1|^{\frac 15} - 2|p|^{\frac 15} + |p-1|^{\frac 15}\right)^3.\eeq  Hence, we have that
\beq{var_G_final} \lim_{n \to \infty} \left< v_n, D^3 G_n \right>_{\hten^{\otimes 3}} = 75 \kappa_3 \int_0^t f^{(5)} (B_s)^2 ds.\eeq

\medskip
\subsection{Proof of Theorem 3.3.}
By Sections 3.3, the term \req{simps_err} is dominated in probability by $\frac{1}{2880}(F_n + G_n)$.  By the results of Section 3.4, the vector $(F_n, G_n)$ satisfies Theorem 2.3, that is, $(F_n, G_n)$ converges stably as $n \to \infty$ to a mean-zero Gaussian random vector $(F_\infty, G_\infty)$ with independent components, whose variances are given by \req{var_F_final} and \req{var_G_final}, respectively.  It follows that $F_n + G_n$ converges in distribution to a centered Gaussian random variable with variance
\[ s^2 = \frac{5!}{2^5}\kappa_5 \int_0^t f^{(5)}(B_s)^2~ds + 75\kappa_3\int_0^t f^{(5)}(B_s)^2~ds = \beta^2 \int_0^t f^{(5)}(B_s)^2~ds,\]
where $\beta^2 = (5!)2^{-5}\kappa_5 + 75\kappa_3$.  The result of Theorem 3.3 then follows from the It\^o isometry.  This concludes the proof.
   
\section{Proof of Technical Lemmas}%%%%%%%%%%
\subsection{Proof of Lemma 3.2}
To simplify notation, let $Y_j := \phi(\Btilde)$.  Note that by (B.1), we have $\| \Delta B_{\frac jn}\|_{L^2(\Omega)} = \| \partial_{\frac jn}\|_\hten = n^{-H}$.  For Hermite polynomials $H_r(x)$, $r \ge 1$, it can be shown by induction on the relation $H_{q+1}(x) = xH_q(x) -qH_{q-1}(x)$ that 
\[ x^r = \sum_{p=0}^{\left\lfloor \frac r2 \right\rfloor} C(r,p)H_{r-2p}(x),\]
where each $C(r,p)$ is an integer constant.  From Section 2.1, we use \req{Hmap} with $x = \Delta B_{\frac jn}/\| \Delta B_{\frac jn}\|_{L^2(\Omega)} = n^{H} \Delta B_{\frac jn}$ to write
\[ H_r \left( n^{H}\Delta B_{\frac jn}\right) = \delta^r\left( n^{rH}\partial_{\frac jn}^{\otimes r}\right).\]
It follows that
\[ n^{rH}\Delta B_{\frac jn}^r = \sum_{p=0}^{\left\lfloor \frac r2 \right\rfloor} C(r,p)H_{r-2p}(n^{H}\Delta B_{\frac jn}) = \sum_{p=0}^{\left\lfloor \frac r2 \right\rfloor} C(r,p)\delta^{r-2p}\left( n^{(r-2p)H}\partial_{\frac jn}^{\otimes r-2p}\right),\]
which implies
\[ \Delta B_{\frac jn}^r = \sum_{p=0}^{\left\lfloor \frac r2 \right\rfloor} C(r,p)n^{-2pH}\delta^{r-2p}\left( \partial_{\frac jn}^{\otimes r-2p}\right).\]
With this representation for $\Delta B_{j/n}^r$, we then have
\begin{multline}\label{g_DB_r} {\mathbb E}\left[ \left( \sum_{j=0}^{\Nt -1} Y_j \Delta B_{\frac jn}^r\right)^2\right]
\\ = \sum_{p, p' =0}^{\lfloor \frac r2 \rfloor} C(r,p)C(r,p')n^{-2H(p+p')}\sum_{j,k=0}^{\Nt -1}{\mathbb E}\left[Y_j Y_k\delta^{r-2p}\left(\partial_{\frac jn}^{\otimes r-2p}\right)\delta^{r-2p'}\left(\partial_{\frac kn}^{\otimes r-2p'}\right)\right]
\\ \le \sum_{p, p' =0}^{\lfloor \frac r2 \rfloor} |C(r,p)C(r,p')|n^{-2H(p+p')}\sum_{j,k=0}^{\Nt -1}\left|{\mathbb E}\left[Y_j Y_k\delta^{r-2p}\left(\partial_{\frac jn}^{\otimes r-2p}\right)\delta^{r-2p'}\left(\partial_{\frac kn}^{\otimes r-2p'}\right)\right]\right|.\end{multline}
By Lemma 2.1.d, the product
\[\delta^{r-2p}\left(\partial_{\frac jn}^{\otimes r-2p}\right)\delta^{r-2p'}\left(\partial_{\frac kn}^{\otimes r-2p'}\right)\]
consists of terms of the form
\beq{rpz_int} C\delta^{2r-2(p+p')-2z}\left(\partial_{\frac jn}^{\otimes r-2p-z} \otimes \partial_{\frac kn}^{\otimes r-2p'-z}\right)\left< \partial_{\frac jn}, \partial_{\frac kn}\right>_\hten^z,\eeq
where $z\ge 0$ is an integer satisfying $2r-2(p+p')-2z \ge 0$.  Using \req{rpz_int}, we can write that \req{g_DB_r} consists of nonnegative terms of the form
\beq{delta_2r_2p} Cn^{-2H(p+p')}\sum_{j,k=0}^{\Nt -1}\left|{\mathbb E}\left[Y_j Y_k\delta^{2r-2(p+p')-2z}\left(\partial_{\frac jn}^{\otimes r-2p-z} \otimes \partial_{\frac kn}^{\otimes r-2p'-z}\right)\left< \partial_{\frac jn}, \partial_{\frac kn}\right>_\hten^z\right]\right|.\eeq
To address terms of this type, suppose first that $z \ge 1$.  Lemma 2.1.c implies that
\begin{multline*}\left\| \delta^{2r-2(p+p')-2z}\left(\partial_{\frac jn}^{\otimes r-2p-z} \otimes \partial_{\frac kn}^{\otimes r-2p'-z}\right)\right\|_{L^2(\Omega)}
\le C\left( \| \partial_{\frac jn}\|_\hten^{r-2p-z}  \|\partial_{\frac kn}\|_\hten^{r-2p'-z}\right)\\
\le C\left\| \partial_{\frac 1n}\right\|_\hten^{2r-2(p+p')-2z} = Cn^{-2H(r-p-p'-z)}.\end{multline*}
Hence, for $z \ge 1$, \req{delta_2r_2p} is bounded by
\begin{multline*} Cn^{-2H(p+p')}\sup_{0\le j \le \Nt} \left\| Y_j\right\|^2_{L^2(\Omega)} \left\|\partial_{\frac 1n}\right\|_\hten^{2r-2(p+p')-2z} \sum_{j,k=0}^{\Nt -1} \left|\left< \partial_{\frac jn}, \partial_{\frac kn}\right>_\hten^z\right|\\
\le C\sup_{0\le j \le \Nt}\left\| Y_j\right\|_{{\mathbb D}^{2r,2}}^2\Nt n^{-2rH},\end{multline*}
which follows from Lemma 2.6.c.

On the other hand, for the terms with $z=0$, by \req{duality} we have
\begin{multline}\label{z_zero}
{\mathbb E}\left[Y_j Y_k\delta^{2r-2(p+p')}\left(\partial_{\frac jn}^{\otimes r-2p} \otimes \partial_{\frac kn}^{\otimes r-2p'}\right)\right]\\
={\mathbb E}\left< D^{2r-2(p+p')}Y_jY_k, \partial_{\frac jn}^{\otimes r-2p} \otimes \partial_{\frac kn}^{\otimes r-2p'}\right>_{\hten^{\otimes 2r-2(p+p')}}.\end{multline}
By definition of the Malliavin derivative and Leibniz rule, $D^{2r-2(p+p')}Y_jY_k$ consists of terms of the form
$D^aY_j\otimes D^bY_k$, 
where $a+b = 2r-2(p+p')$.  Without loss of generality, we may assume $b\ge 1$.  By assumptions on $\phi$ and the definition of the Malliavin derivative, we know that $D^b Y_k = \phi^{(b)}(\widetilde B_{k/n})\widetilde \varepsilon_{k/n}^{\otimes b}$, 
 and we know that for each $b \le 2r$, $D^bY_k \in L^2(\Omega; \hten^{\otimes b})$.
  It follows that we can write,
\begin{multline*}
\left| {\mathbb E}\left< D^aY_j\otimes D^b Y_k , \partial_{\frac jn}^{\otimes r-2p} \otimes \partial_{\frac kn}^{\otimes r-2p'}\right>_{\hten^{\otimes a+b}}\right|\\
\le C \|Y_j\|_{{\mathbb D}^{2r, 2}}\|Y_k\|_{{\mathbb D}^{2r, 2}} \left|\left<\widetilde \varepsilon_{\frac jn}, \partial_{\frac jn}\right>_\hten^\phi\right|~\left|\left<\widetilde \varepsilon_{\frac jn}, \partial_{\frac kn}\right>_\hten^{a-\phi}\right|\\ \times\left|\left<\widetilde \varepsilon_{\frac kn}, \partial_{\frac jn}\right>_\hten^\psi\right|~\left|\left<\widetilde \varepsilon_{\frac kn}, \partial_{\frac kn}\right>_\hten^{b-\psi}\right|,
\end{multline*}
for integers $0\le \phi\le a$, $0\le\psi\le b$.  Without loss of generality, we may assume $\psi \ge 1$, and by implication $b\ge 1$.  Then using (B.4), 
\begin{equation*}
\left| {\mathbb E}\left< D^aY_j D^b Y_k , \partial_{\frac jn}^{\otimes r-2p} \otimes \partial_{\frac kn}^{\otimes r-2p'}\right>_{\hten^{\otimes a+b}}\right|\le C\sup_{0\le j\le \Nt} \| Y_j\|^2_{{\mathbb D}^{2r,2}} n^{-2H(a+b-1)}\left|\left<\widetilde \varepsilon_{\frac kn}, \partial_{\frac jn}\right>_\hten\right|.
\end{equation*} 
Thus, for each pair $(a,b)$, the corresponding term of \req{delta_2r_2p} is bounded by
\begin{align*}
Cn^{-2H(p+p')}\sum_{j,k=0}^{\Nt -1}&\left|{\mathbb E}\left[Y_j Y_k\delta^{2r-2(p+p')}\left(\partial_{\frac jn}^{\otimes r-2p} \otimes \partial_{\frac kn}^{\otimes r-2p'}\right)\right]\right|\\
&\le Cn^{-2H(p+p'+a+b-1)} \sup_{0\le j\le \Nt} \| Y_j\|^2_{{\mathbb D}^{2r,2}} \sum_{j,k=0}^{\Nt -1}\left|\left<\widetilde \varepsilon_{\frac kn}, \partial_{\frac jn}\right>_\hten\right|\\
&\le  Cn^{-2H(p+p'+a+b-1)} \sup_{0\le j\le \Nt} \| Y_j\|^2_{{\mathbb D}^{2r,2}} \sum_{j,k=0}^{\Nt -1} \left|\left<\widetilde \varepsilon_{\frac kn}, \partial_{\frac jn}\right>_\hten\right|.
\end{align*}
By Lemma 2.6.a,
\[ \sum_{j=0}^{\Nt -1} \left|\left<\widetilde \varepsilon_{\frac kn}, \partial_{\frac jn}\right>_\hten\right| \le C\Nt^{2H}n^{-2H}\le C\]
for all $0\le k\le \Nt$, so that
\begin{multline*}  Cn^{-2H(p+p'+a+b-1)} \sup_{0\le j\le \Nt} \| Y_j\|^2_{{\mathbb D}^{2r,2}} \sum_{k=0}^{\Nt -1}\left\{\sup_{0\le k \le\Nt}\sum_{j=0}^{\Nt -1} \left|\left<\widetilde \varepsilon_{\frac kn}, \partial_{\frac jn}\right>_\hten\right|\right\}\\ \le C\sup_{0\le j\le \Nt} \| Y_j\|^2_{{\mathbb D}^{2r,2}}\Nt n^{-2H(p+p'+a+b-1)},\end{multline*}
where $p+p'+a+b-1 = 2r - (p+p')-1 \ge r$, since $p+p' +1 \le 2\left\lfloor \frac r2\right\rfloor +1 \le r$, for odd integer $r$.
This concludes the proof.

\subsection{Proof of Lemma 3.6.}
For $\theta \in\{ 0, 2\}$ define
\[ w_n(\theta) = n^{-\theta H}\sum_{j=0}^{\Nt - 1} f^{(5)}(\Btilde) \partial_{\frac jn}^{\otimes 5-\theta};\; \text{ and }\; \Phi_n(\theta) = \delta^{5-\theta}(w_n(\theta)).\]
This allows us to write $u_n = w_n(0)$, $F_n = \Phi_n(0)$, $v_n = 10w_n(2)$, and $G_n = 10\Phi_n(2)$.  Following Remark 2.4, we may assume that $h \in\hten^{\otimes 5-\theta}$ has the form $\varepsilon_{t_1} \otimes \cdots \otimes \varepsilon_{t_{5-\theta}}$, for some set of times $\{ t_1, \dots, t_{5-\theta}\}$ in $[0,T]^{5-\theta}$.  Then for (a), using (B.4) and Lemma 2.6.a,
\begin{align*}
{\mathbb E}\left| \left< w_n(\theta) , h \right>_{\hten^{\otimes 5-\theta}}\right| &= n^{-\theta H}{\mathbb E}\left| \sum_{j=0}^{\Nt -1} \left< f^{(5)}(\Btilde)\partial_{\frac jn}^{\otimes 5-\theta}, \varepsilon_{t_1} \otimes \cdots \otimes \varepsilon_{t_{5-\theta}}\right>_{\hten^{\otimes 5-\theta}}\right|\\
&\le n^{-\theta H} {\mathbb E}\left[ \sup_{s\in[0,t]}\left| f^{(5)}(B_s)\right|\right] \sum_{j=0}^{\Nt -1} \prod_{k=1}^{5-\theta} \left| \left< \partial_{\frac jn}, \varepsilon_{t_k}\right>_\hten \right|\\
%&\le Cn^{-8H + \theta H} \sum_{j=0}^{\Nt -1} \left|\left< \partial_{\frac jn}, \varepsilon_{t_1}\right>_\hten \right| \\
&\le Cn^{-(8-\theta)H} \le Cn^{-6H},\end{align*}
where the last inequality follows because $\theta \le 2$.

Next, for (b), consider integers $0 \le a_i < 5-\theta$, $0\le s\le r < 5-\theta$, $r \ge 1$ and $q$, such that $s \le r$, $1 \le a_1 + \cdots + a_r < 5-\theta$ and $q = 5-\theta - (a_1 + \cdots + a_r) \ge 1$.  We have
\begin{multline*}
{\mathbb E} \left| \left< w_n(\theta), \bigotimes_{i=1}^s D^{a_i}F_n \bigotimes_{i=s+1}^r  D^{a_i}G_n \otimes h\right>_{\hten^{\otimes 5-\theta}}\right|\\
\le n^{-\theta H} {\mathbb E} \sum_{j=0}^{\Nt -1} \left| f^{(5)} (\Btilde) \prod_{i=1}^s \left< \partial_{\frac jn}^{\otimes a_i}, D^{a_i}F_n\right>_{\hten^{\otimes a_i}}\left(\prod_{i=s+1}^r \left< \partial_{\frac jn}^{\otimes a_i}, D^{a_i}G_n\right>_{\hten^{\otimes a_i}}\right)\left< \partial_{\frac jn}^{\otimes q} , h\right>_{\hten^{\otimes q}}\right|.\end{multline*}
Using (B.1), Lemma 3.5, and Lemma 2.6.a, this is bounded by
\begin{multline*}
n^{-\theta H} \sup_{0\le j \le \Nt}\left\| f^{(5)}(\Btilde)\right\|_{L^p(\Omega)} \prod_{i=1}^r \sup_j \left\| \partial_{\frac jn}^{\otimes a_i}\right\|_{\hten^{\otimes a_i}} \prod_{i=1}^s \left\| D^{a_i} F_n\right\|_{L^p(\Omega; \hten^{\otimes a_i})}\\ \times \prod_{i=s+1}^r \left\| D^{a_i} G_n\right\|_{L^p(\Omega; \hten^{\otimes a_i})}  \sum_{j=0}^{\Nt -1} \left| \left< \partial_{\frac jn}^{\otimes q}, h \right>_{\hten^{\otimes q}}\right|
\le Cn^{-(3+q)H},\end{multline*}
where $p = r+1$.% and by (B.1) and Gaussian property, we have for all $i,j$, \[\left\| \partial_{\frac jn}^{\otimes a_i}\right\|_{L^p(\Omega; \hten^{\otimes a_i})} \le C\left\| \partial_{\frac 1n}\right\|_\hten^{a_i} \le Cn^{-a_i H}.\]
%Note that in the above estimate, by Lemma 2.6.a,
%\[ \sum_{j=0}^{Nt -1}\left|\left< \partial_{\frac jn}^{\otimes q}, h\right>_{\hten^{\otimes q}}\right| \le \max_{1\le\ell\le q} \sum_{j=0}^{\Nt -1} \left|\left< \partial_{\frac jn}, \varepsilon_{t_\ell}\right>_\hten^q\right| \le C\Nt^{2H} n^{-2qH}.\]

\medskip
%%%%%%%%%%%%%%%%%%%<u_n, DF x DF> etc.
For (c), we want to consider terms of the form
\[ {\mathbb E}\left| \left< w_n(\theta_0), \bigotimes_{i=1}^r D^{a_i} \Phi_n(\theta_i)\right>_{\hten^{\otimes 5-\theta_0}}\right|,\]
where $\theta_i \in \{0,2\}$, $2 \le r \le 5-\theta_0$, $0 \le a_i \le 4-\theta_0$, and $a_1 + \cdots + a_r = 5-\theta_0$.
For example, the term
\[ \left< u_n, D^3 F_n \otimes D^2 G_n \right>_{\hten^{\otimes 3}}\]
corresponds to the case $(\theta_0, \theta_1, \theta_2) = (0,0,2)$, $a_1=3$, $a_2=2$.  We will show that terms of this type tend to zero in $L^2(\Omega)$ as $n \to \infty$.  Using the above definitions for $w_n(\theta_i),~\Phi_n(\theta_i)$, we have
\begin{multline}\label{wDrexpand}
{\mathbb E}\left[\left< w_n(\theta_0), \bigotimes_{i=1}^r D^{a_i} \Phi_n(\theta_i)\right>_{\hten^{\otimes 5-\theta_0}}^2\right]\\
 = n^{-2H(\theta_0 + \cdots + \theta_r)}{\mathbb E}~\sum_{p,p'=0}^{\Nt -1}\sum_{j_1, \dots, j_r = 0}^{\Nt -1}\sum_{k_1, \dots, k_r = 0}^{\Nt -1}
\left< f^{(5)}(\widetilde B_{\frac{p}{n}})\partial_{\frac{p}{n}}^{\otimes 5-\theta_0} , \bigotimes_{i=1}^r D^{a_i}\delta^{5-\theta_i}\left( f^{(5)}(\widetilde B_{\frac{j_i}{n}})\partial_{\frac{j_i}{n}}^{\otimes 5-\theta_i}\right)\right>_{\hten^{\otimes 5-\theta_0}}\\
\times\left< f^{(5)}(\widetilde B_{\frac{p'}{n}})\partial_{\frac{p'}{n}}^{\otimes 5-\theta_0} , \bigotimes_{i=i}^r D^{a_1}\delta^{5-\theta_i}\left( f^{(5)}(\widetilde B_{\frac{k_i}{n}})\partial_{\frac{k_i}{n}}^{\otimes 5-\theta_i}\right)\right>_{\hten^{\otimes 5-\theta_0}}.
\end{multline}
By Lemma 2.1.b,
\begin{multline*} D^{a_i}\delta^{5-\theta_i}\left( f^{(5)}(\widetilde B_{\frac{j_i}{n}})\partial_{\frac{j_i}{n}}^{\otimes 5-\theta_i}\right)\\
= \sum_{\ell_i=0}^{(5-\theta_i)\wedge a_i} {\ell_i!}\binom{5-\theta_i}{\ell_i}\binom{a_i}{\ell_i} \delta^{5-\theta_i -\ell_i}\left(f^{(5+a_i-\ell_i)}(\widetilde B_{\frac{j_i}{n}})\partial_{\frac{j_i}{n}}^{\otimes 5-\theta_i -\ell_i}\right)\partial_{\frac{j_i}{n}}^{\otimes \ell_i}\otimes \widetilde \varepsilon_{\frac{j_i}{n}}^{\otimes a_i - \ell_i}.\end{multline*}
Applying this to each term, we can expand the inner product
\[\left< f^{(5)}(\widetilde B_{\frac pn})\partial_{\frac pn}^{\otimes 5-\theta_0} , D^{a_1} \delta^{5-\theta_1}\left( f^{(5)}(\widetilde B_{\frac{j_1}{n}})\partial_{\frac{j_1}{n}}^{\otimes 5-\theta_1}\right)\otimes\cdots\otimes D^{a_r}\delta^{5-\theta_r}\left( f^{(5)}(\widetilde B_{\frac{j_r}{n}})\partial_{\frac{j_r}{n}}^{\otimes 5-\theta_r}\right)\right>_{\hten^{\otimes 5-\theta_0}}\]
into terms of the form
\begin{multline*}
C_\ell f^{(5)}(\widetilde B_{\frac pn}) \delta^{b_1}\left( f^{(\lambda_1)}(\widetilde B_{\frac{j_1}{n}})\partial_{\frac{j_1}{n}}^{\otimes b_1}\right) \cdots\delta^{b_r}\left( f^{(\lambda_r)}(\widetilde B_{\frac{j_r}{n}})\partial_{\frac{j_r}{n}}^{\otimes b_r}\right)\\
\times \left< \partial_{\frac pn}, \partial_{\frac{j_1}{n}}\right>_\hten^{\ell_1}\left< \partial_{\frac pn}, \widetilde \varepsilon_{\frac{j_1}{n}}\right>_\hten^{a_1 - \ell_1}\cdots \left< \partial_{\frac pn}, \partial_{\frac{j_r}{n}}\right>_\hten^{\ell_r}\left< \partial_{\frac pn}, \widetilde \varepsilon_{\frac{j_r}{n}}\right>_\hten^{a_r - \ell_r},
\end{multline*}
where $C_\ell = C_\ell (\ell_1, \dots, \ell_r)$ is an integer constant, each $b_i = 5-\theta_i - \ell_i$, and each $\lambda_i = 5+a_i - \ell_i$.
It follows that \req{wDrexpand} is a sum of terms of the form
\begin{multline}\label{wDrCC}
C_\ell C_{\ell'} n^{-2H(\theta_1 + \cdots + \theta_r)} {\mathbb E} \sum_{p,p'=0}^{\Nt -1} f^{(5)}(\widetilde B_{\frac{p}{n}})f^{(5)}(\widetilde B_{\frac{p'}{n}})\\
\times \left( \sum_{j_1 = 0}^{\Nt -1} \delta^{b_1}\left( f^{(\lambda_1)}(\widetilde B_{\frac{j_1}{n}}) \partial_{\frac{j_1}{n}}^{\otimes b_1}\right)\left< \partial_{\frac pn}, \partial_{\frac{j_1}{n}}\right>_\hten^{\ell_1}\left< \partial_{\frac pn}, \widetilde \varepsilon_{\frac{j_1}{n}}\right>_\hten^{a_1 - \ell_1}\right)\\
\times \cdots \times\left( \sum_{k_r = 0}^{\Nt -1} \delta^{b_r'}\left( f^{(\lambda_r')}(\widetilde B_{\frac{k_r}{n}}) \partial_{\frac{k_r}{n}}^{\otimes b_r'}\right)\left< \partial_{\frac{p'}{n}}, \partial_{\frac{k_r}{n}}\right>_\hten^{\ell_r'}\left< \partial_{\frac{p'}{n}}, \widetilde \varepsilon_{\frac{k_r}{n}}\right>_\hten^{a_r - \ell_r'}\right).\end{multline}
For $0\le j_1, \dots, j_r \le \Nt$ we have the estimate
\begin{multline*}
\sum_{p=0}^{\Nt -1} \left| \left< \partial_{\frac pn}, \partial_{\frac{j_1}{n}}\right>_\hten^{\ell_1}\left< \partial_{\frac pn}, \widetilde \varepsilon_{\frac{j_1}{n}}\right>_\hten^{a_1 - \ell_1}\cdots \left< \partial_{\frac pn}, \partial_{\frac{j_r}{n}}\right>_\hten^{\ell_r}\left< \partial_{\frac pn}, \widetilde \varepsilon_{\frac{j_r}{n}}\right>_\hten^{a_r - \ell_r}\right|\\
\le \sup_{\cal I} \sum_{p=0}^{\Nt -1} \prod_{i=1}^r\left|\left< \partial_{\frac pn}, \partial_{\frac{j_i}{n}}\right>_\hten^{\ell_i}\left<\partial_{\frac pn}, \widetilde \varepsilon_{\frac{j_i}{n}}\right>_\hten^{a_i-\ell_i}\right|,\end{multline*}
where ${\cal I} = \{ 0 \le j_1, \dots, j_r\le \Nt\}$.  By Lemma 2.6.a and/or 2.6.c, this is bounded by $Cn^{-2H(5-\theta_0)}$ if $\ell_1 + \cdots + \ell_r \ge 1$, and bounded by $Cn^{-2H(5-\theta_0 -1)} = Cn^{-2H(4-\theta_0)}$ if and only if $\ell_1 = \cdots = \ell_r =0$.  Hence, we can write
\beq{wDrpp} \sup_{\cal I, I'} \sum_{p,p'=0}^{\Nt -1} \left| \left< \partial_{\frac pn}, \partial_{\frac{j_1}{n}}\right>_\hten^{\ell_1}\cdots \left<\partial_{\frac{p'}{n}}, \widetilde \varepsilon_{\frac{k_r}{n}}\right>_\hten^{a_r-\ell_r'}\right|
\le Cn^{-\Lambda H},\eeq
where $4H(4-\theta_0) \le \Lambda \le 4H(5-\theta_0)$.

It follows that terms of the form \req{wDrCC} can be bounded in absolute value by
\begin{multline*}
Cn^{-2H(\theta_0 +\cdots +\theta_r)} \sup_{0\le p \le\Nt} \| f^{(5)}(\widetilde B_{\frac pn})\|^2_{L^{4r+2}(\Omega)}
\sup_{\cal I, I'} \sum_{p,p'=0}^{\Nt -1} \left| \left< \partial_{\frac pn}, \partial_{\frac{j_1}{n}}\right>_\hten^{\ell_1}\cdots \left<\partial_{\frac{p'}{n}}, \widetilde \varepsilon_{\frac{k_r}{n}}\right>_\hten^{a_r-\ell_r'}\right|\\
\times \prod_{i=1}^r \left\| \sum_{j_i=0}^{\Nt -1} \delta^{b_i}\left( f^{(\lambda_i)}(\widetilde B_{\frac{j_i}{n}})\partial_{\frac{j_i}{n}}^{\otimes b_i}\right)\right\|_{L^{2r+1}(\Omega)}\left\| \sum_{k_i=0}^{\Nt -1} \delta^{b_i'}\left( f^{(\lambda_i')}(\widetilde B_{\frac{k_i}{n}})\partial_{\frac{k_i}{n}}^{\otimes b_i'}\right)\right\|_{L^{2r+1}(\Omega)}.
\end{multline*}
By \req{wDrpp} and Lemma 3.5, this is bounded by
\[ C\Nt^r n^{-2H(\theta_0 +\cdots +\theta_r) - \Lambda H - H(b_1 +\cdots + b_r + b_1'+\cdots + b_r')}.\]
We have $\Lambda \ge 4H(4-\theta_0)$, and 
\[b_1 + \cdots + b_r = 5r - (\theta_1 +\cdots + \theta_r) - (\ell_1 + \cdots + \ell_r).\]
Since $\ell_i \le a_i$ for each $i$, then $\ell_1 + \cdots +\ell_r \le a_1 + \cdots + a_r = 5-\theta_0$, it follows that the exponent
\begin{multline*} 2H(\theta_0 + \cdots + \theta_r) + \Lambda H + H(b_1 + \cdots + b_r + b_1' + \cdots + b_r')\\
\ge 2H(\theta_0 + \cdots + \theta_r) + 4H(4-\theta_0) + H(10r -2(\theta_1 +\cdots +\theta_r) - 2(5-\theta_0))\\ \ge 16H + 10(r-1)H \ge 10rH + 6H.\end{multline*}
Hence, we have an upper bound of
\[ C\Nt^r n^{-10rH-6H} \le Cn^{-6H}\]
for each term of the form \req{wDrCC}, so this term tends to zero in $L^2(\Omega)$, and we have (c).
This concludes the proof of Lemma 3.6. 
$\square$

\subsection{Proof of Lemma 3.7.}
Starting with (a), Lemma 2.1.b gives
\begin{align*}
{\mathbb E} \left| \left< u_n, D^5 G_n \right>_{\hten^{\otimes 5}}\right| &=
n^{-2H}{\mathbb E}\left| \sum_{i=0}^3 \binom 5i \binom 3i i! \sum_{j,k=0}^{\Nt -1} \left< f^{(5)}(\Btilde)\partial_{\frac jn}^{\otimes 5}, \delta^{3-i}\left( f^{(10-i)}(\widetilde B_{\frac kn})\partial_{\frac kn}^{\otimes 3-i}\right)\partial_{\frac kn}^{\otimes i}\otimes \widetilde \varepsilon_{\frac kn}^{\otimes 5-i}\right>_{\hten^{\otimes 5}}\right|\\
&\le Cn^{-2H}\sum_{i=0}^3 \sup_{0\le j \le \Nt} \left\| f^{(5)}(\Btilde)\right\|_{L^2(\Omega)} \sup_{0\le k \le \Nt}\left\| \delta^{3-i}\left( f^{(10-i)}(\widetilde B_{\frac kn})\partial_{\frac kn}^{\otimes 3-i}\right)\right\|_{L^2(\Omega)}\\
&\qquad \times \sum_{j,k=0}^{\Nt -1} \left| \left< \partial_{\frac jn}, \partial_{\frac kn}\right>_\hten^i \left< \partial_{\frac jn}, \widetilde \varepsilon_{\frac kn}\right>_\hten^{5-i}\right|.
\end{align*}
By moderate growth conditions and \req{Meyer}, we have
$\left\| f^{(5)}(\Btilde)\right\|_{L^2(\Omega)} \le C$ and $\left\| \delta^{3-i}\left( f^{(10-i)}(\widetilde B_{\frac kn})\partial_{\frac kn}^{\otimes 3-i}\right)\right\|_{L^2(\Omega)} \le C\|\partial_{\frac 1n}\|_\hten^{3-i} = Cn^{-(3-i)H}$; so we have terms of the form
\[ Cn^{-(5-i)H} \sum_{j,k=0}^{\Nt -1} \left| \left< \partial_{\frac jn}, \partial_{\frac kn}\right>_\hten^i \left< \partial_{\frac jn}, \widetilde \varepsilon_{\frac kn}\right>_\hten^{5-i}\right|.\]
If $i > 0$, then (B.4) and Lemma 2.6.c give an estimate of
\begin{multline*}
Cn^{-(5-i)H} \sum_{j,k=0}^{\Nt -1} \left| \left< \partial_{\frac jn}, \partial_{\frac kn}\right>_\hten^i \left< \partial_{\frac jn}, \widetilde \varepsilon_{\frac kn}\right>_\hten^{5-i}\right| \\
\le Cn^{-(15-3i)H} \sum_{j,k=0}^{\Nt -1} \left| \left< \partial_{\frac jn}, \partial_{\frac kn}\right>_\hten^i\right| \le C\Nt n^{-(15-3i)H} \le Cn^{-2H},\end{multline*}
because $i \le 3$.  On the other hand, if $i =0$, then by (B.4) and Lemma 2.6.a,
	\begin{multline*} Cn^{-5H} \sum_{j,k=0}^{\Nt -1} \left| \left< \partial_{\frac jn}, \widetilde \varepsilon_{\frac kn}\right>_\hten^5\right| \le Cn^{-5H}\sum_{k=0}^{\Nt -1} \left\{\sup_{0\le k \le \Nt}\sum_{j=0}^{\Nt -1}\left| \left< \partial_{\frac jn}, \widetilde \varepsilon_{\frac kn}\right>_\hten^5\right|\right\}\\ \le C\Nt n^{-13H} \le Cn^{-3H},\end{multline*}
hence (a) is proved.

For (b), again using Lemma 2.1.b we can write
\begin{align*}
{\mathbb E} \left| \left< v_n, D^3 F_n \right>_{\hten^{\otimes 3}}\right| &=
n^{-2H}{\mathbb E}\left| \sum_{i=0}^3 \binom 5i \binom 3i i! \sum_{j,k=0}^{\Nt -1} \left< f^{(5)}(\Btilde)\partial_{\frac jn}^{\otimes 3}, \delta^{5-i}\left( f^{(8-i)}(\widetilde B_{\frac kn})\partial_{\frac kn}^{\otimes 5-i}\right)\partial_{\frac kn}^{\otimes i}\otimes \widetilde \varepsilon_{\frac kn}^{\otimes 3-i}\right>_{\hten^{\otimes 3}}\right|\\
&\le Cn^{-2H}\sum_{i=0}^3 {\mathbb E}\left|\sum_{j,k=0}^{\Nt -1} f^{(5)}(\Btilde)\delta^{5-i}\left( f^{(8-i)}(\widetilde B_{\frac kn})\partial_{\frac kn}^{\otimes 5-i}\right)\left<\partial_{\frac jn}, \partial_{\frac kn}\right>^i_\hten \left< \partial_{\frac jn},  \widetilde \varepsilon_{\frac kn}\right>_\hten^{3-i}\right|.
\end{align*}
We deal with three cases.  First, assume $i=0$.  Then we have a bound of
\begin{multline*}
Cn^{-2H} \sum_{j,k=0}^{\Nt -1} {\mathbb E}\left| f^{(5)}(\Btilde)\delta^{5}\left( f^{(8)}(\widetilde B_{\frac kn})\partial_{\frac kn}^{\otimes 5}\right)\right|~\left| \left<\partial_{\frac jn}, \widetilde \varepsilon_{\frac kn}\right>^3_\hten\right|\\
\le Cn^{-2H}\sup_{0\le j \le \Nt} \left\| f^{(5)}(\Btilde)\right\|_{L^2(\Omega)}\sup_{0\le k \le \Nt}\left\|\delta^{5}\left( f^{(8)}(\widetilde B_{\frac kn})\partial_{\frac kn}^{\otimes 5}\right)\right\|_{L^2(\Omega)} \sup_{j,k} \left|\left< \partial_{\frac jn}, \widetilde \varepsilon_{\frac kn}\right>^2_\hten\right|\\
\times \sum_{k=0}^{\Nt -1}\left\{\sup_{0\le k \le \Nt} \sum_{j=0}^{\Nt -1} \left|\left< \partial_{\frac jn}, \widetilde \varepsilon_{\frac kn}\right>_\hten\right|\right\} \le C\Nt n^{-11H} \le Cn^{-H},
\end{multline*}
where, as above, we use the estimates $\left\| f^{(5)}(\Btilde)\right\|_{L^2(\Omega)} \le C$ and $\left\| \delta^{5}\left( f^{(8)}(\widetilde B_{\frac kn})\partial_{\frac kn}^{\otimes 5}\right)\right\|_{L^2(\Omega)} \le Cn^{-5H}$; and 
\[ \sup_{j,k}\left|\left< \partial_{\frac jn}, \widetilde \varepsilon_{\frac kn}\right>^2_\hten\right|
\sum_{k=0}^{\Nt -1} \sum_{j=0}^{\Nt -1} \left|\left< \partial_{\frac jn}, \widetilde \varepsilon_{\frac kn}\right>_\hten\right| \le C\Nt n^{-4H}\] follows from (B.4) and Lemma 2.6.a.

The next case is for $i = 1$ or $i=2$.  Using similar estimates we have
\begin{multline*} 
Cn^{-2H} \sum_{j,k=0}^{\Nt -1} {\mathbb E}\left|f^{(5)}(\Btilde)\delta^{5-i}\left( f^{(8-i)}(\widetilde B_{\frac kn})\partial_{\frac kn}^{\otimes 5-i}\right)\right|~\left|\left<\partial_{\frac jn}, \partial_{\frac kn}\right>^i_\hten \left< \partial_{\frac jn},  \widetilde \varepsilon_{\frac kn}\right>_\hten^{3-i}\right|\\
\le Cn^{-2H}\sup_{0\le j \le \Nt} \left\| f^{(5)}(\Btilde)\right\|_{L^2(\Omega)}\sup_{0\le k \le \Nt}\left\|\delta^{5-i}\left( f^{(8-i)}(\widetilde B_{\frac kn})\partial_{\frac kn}^{\otimes 5-i}\right)\right|_{L^2(\Omega)} \sup_{j,k} \left|\left< \partial_{\frac jn}, \widetilde \varepsilon_{\frac kn}\right>^{3-i}_\hten\right|\\
\times \sum_{k=0}^{\Nt -1} \sum_{j=0}^{\Nt -1} \left|\left< \partial_{\frac jn}, \partial_{\frac kn}\right>_\hten^i\right| \le C\Nt n^{-(7-i+6)H} \le Cn^{-H},
\end{multline*}
because $7-i+6 \ge 11$ for $i \le 2$.

For the case $i=3$, we will use a different estimate, and show that the term with $i=3$ vanishes in $L^2(\Omega)$.  Using Lemma 2.1.d we have,
\begin{align*}
&{\mathbb E}\left[\left( n^{-2H} \sum_{j,k=0}^{\Nt -1} f^{(5)}(\Btilde)\delta^2\left( f^{(5)}(\widetilde B_{\frac kn})\partial_{\frac kn}^{\otimes 2}\right)\left< \partial_{\frac jn}, \partial_{\frac kn}\right>_\hten^3\right)^2\right]\\
&\quad = n^{-4H}\sum_{j,j',k,k'=0}^{\Nt -1}{\mathbb E}\left[f^{(5)}(\Btilde)f^{(5)}(\widetilde B_{\frac{j'}{n}})\delta^2\left( f^{(5)}(\widetilde B_{\frac kn})\partial_{\frac kn}^{\otimes 2}\right)\delta^2\left( f^{(5)}(\widetilde B_{\frac{k'}{n}})\partial_{\frac{k'}{n}}^{\otimes 2}\right)\left< \partial_{\frac jn}, \partial_{\frac kn}\right>_\hten^3\left< \partial_{\frac{j'}{n}}, \partial_{\frac{k'}{n}}\right>_\hten^3\right]\\
&\quad = n^{-4H}\sum_{p=0}^2 {\binom 2p}^2p! \sum_{j,j',k,k'} {\mathbb E}\left[ g(j,j') \delta^{4-2p}\left(g(k,k')\partial_{\frac kn}^{\otimes 2-p} \otimes \partial_{\frac{k'}{n}}^{\otimes 2-p}\right)\right]\left< \partial_{\frac kn}, \partial_{\frac{k'}{n}}\right>_\hten^p\left< \partial_{\frac jn}, \partial_{\frac kn}\right>_\hten^{3}\left< \partial_{\frac{j'}{n}}, \partial_{\frac{k'}{n}}\right>_\hten^{3},
\end{align*}
where $g(j,j') = f^{(5)}(\Btilde)f^{(5)}(\widetilde B_{\frac{j'}{n}})$.  Then by the Malliavin duality \req{duality}, this results in a sum of three terms of the form
\beq{p_cases} Cn^{-4H}\sum_{j,j',k,k'} {\mathbb E}\left[ \left< D^{4-2p}g(j,j'), g(k,k')\partial_{\frac kn}^{\otimes 2-p} \otimes \partial_{\frac{k'}{n}}^{\otimes 2-p}\right>_{\hten^{\otimes 4-2p}}\right]\left< \partial_{\frac kn}, \partial_{\frac{k'}{n}}\right>_\hten^p\left< \partial_{\frac jn}, \partial_{\frac kn}\right>_\hten^{3}\left< \partial_{\frac{j'}{n}}, \partial_{\frac{k'}{n}}\right>_\hten^{3},\eeq
for $p = 0,1,2$.  When the index $p=0$, then ${\mathbb E}\left|\left< D^{4-2p}g(j,j'), g(k,k')\partial_{\frac kn}^{\otimes 2-p} \otimes \partial_{\frac{k'}{n}}^{\otimes 2-p}\right>_{\hten^{\otimes 4-2p}}\right|$ consists of terms of the form
\beq{p_zero_Dgg} {\mathbb E}\left| \left(\frac{\partial^4}{\partial x_1^a \partial x_2^b} \Psi(\Btilde, \widetilde B_{\frac{j'}{n}})\right)g(k,k') \left<\widetilde \varepsilon_{\frac jn}, \partial_{\frac kn}\right>_\hten^a\left<\widetilde \varepsilon_{\frac{j'}{n}}, \partial_{\frac kn}\right>_\hten^{2-a} \left<\widetilde \varepsilon_{\frac jn}, \partial_{\frac{k'}{n}}\right>_\hten^b\left<\widetilde \varepsilon_{\frac{j'}{n}}, \partial_{\frac{k'}{n}}\right>_\hten^{2-b}\right|,\eeq
where $\Psi(x_1, x_2) = f^{(5)}(x_1)f^{(5)}(x_2)$ and $a+b = 4$.  By moderate growth and (B.4), we see that \req{p_zero_Dgg} is bounded by $Cn^{-8H}$, and so for the case $p=0$, \req{p_cases} is bounded in absolute value by
\begin{align*}
Cn^{-12H} \sum_{j,j',k,k'}\left| \left< \partial_{\frac jn}, \partial_{\frac kn}\right>_\hten^{3}\left< \partial_{\frac{j'}{n}}, \partial_{\frac{k'}{n}}\right>_\hten^{3}\right| &= Cn^{-12H}\left( \sum_{j,k=0}^{\Nt -1} \left|\left< \partial_{\frac jn}, \partial_{\frac kn}\right>_\hten^{3}\right|\right)^2\\
&\le C\Nt^2 n^{-24H} \le Cn^{-4H}.
\end{align*}

By a similar estimate, when $p=1$, then 
\[{\mathbb E}\left|\left< D^2 g(j,j') , g(k,k') \partial_{\frac kn}\otimes\partial_{\frac{k'}{n}}\right>_{\hten^{\otimes 2}}\right| \le Cn^{-4H},\]
so that for $p=1$, then \req{p_cases} is bounded in absolute value by
\begin{multline*}
Cn^{-8H} \sum_{j,j',k,k'} \left| \left<\partial_{\frac kn}, \partial_{\frac{k'}{n}}\right>_\hten\left< \partial_{\frac jn}, \partial_{\frac kn}\right>_\hten^{3}\left< \partial_{\frac{j'}{n}}, \partial_{\frac{k'}{n}}\right>_\hten^{3}\right|\\
\le Cn^{-8H} \sup_{k,k'}\left|\left<\partial_{\frac kn}, \partial_{\frac{k'}{n}}\right>_\hten\right|\left( \sum_{j,k=0}^{\Nt -1} \left|\left< \partial_{\frac jn}, \partial_{\frac kn}\right>_\hten^{3}\right|\right)^2
\le C\Nt^2 n^{-22H} \le Cn^{-2H}.\end{multline*}

Last, the term in \req{p_cases} with $p=2$ has the form
\[Cn^{-4H}\sum_{j,j',k,k'} {\mathbb E}\left[ g(j,j') g(k,k')\right]\left<\partial_{\frac kn}, \partial_{\frac{k'}{n}}\right>_\hten^2\left< \partial_{\frac jn}, \partial_{\frac kn}\right>_\hten^{3}\left< \partial_{\frac{j'}{n}}, \partial_{\frac{k'}{n}}\right>_\hten^{3}.\]
This is bounded in absolute value by
\beq{p_cases_2_b}
Cn^{-4H}\sup_{0\le j \le \Nt} \left\| f^{(5)}(\Btilde)\right\|^4_{L^4(\Omega)} \sum_{k,k'=0}^{\Nt -1}\left<\partial_{\frac kn}, \partial_{\frac{k'}{n}}\right>_\hten^2 \sum_{j=0}^{\Nt -1} \left|\left< \partial_{\frac jn}, \partial_{\frac kn}\right>_\hten^{3}\right|\sum_{j'=0}^{\Nt -1} \left|\left< \partial_{\frac{j'}{n}}, \partial_{\frac{k'}{n}}\right>_\hten^{3}\right|.\eeq
By Lemma 2.6.c, for every $0\le k \le \Nt$ we have
\[ \sum_{j=0}^{\Nt -1} \left|\left< \partial_{\frac jn}, \partial_{\frac kn}\right>_\hten^{3}\right| \le Cn^{-6H},\]
hence \req{p_cases_2_b} is bounded by
\[Cn^{-16H} \sum_{k,k'=0}^{\Nt -1}\left<\partial_{\frac kn}, \partial_{\frac{k'}{n}}\right>_\hten^2 \le C \Nt n^{-20H}\le Cn^{-10H}.\]
Lemma 3.7 is proved. $\square$

\end{document}